\documentclass[12pt]{amsart}
\usepackage[margin=1in]{geometry}
\usepackage{enumerate}		
\usepackage{amssymb,amsfonts}
\usepackage{hyperref}
\usepackage{amsthm}
\usepackage{amsmath}
\usepackage{graphicx}
\usepackage{url}

\newtheorem{theorem}{Theorem}
\newtheorem*{theorem*}{Theorem}
\newtheorem{cor}[theorem]{Corollary}
\newtheorem*{cor*}{Corollary}
\newtheorem{prop}[theorem]{Proposition}
\newtheorem*{prop*}{Proposition}
\newtheorem{lemma}[theorem]{Lemma}
\newtheorem*{lemma*}{Lemma}

\newtheorem*{quest*}{Question}

\newtheorem*{conj*}{Conjecture}
\newtheorem{fact}[theorem]{Fact}
\newtheorem*{fact*}{Fact}

\newtheorem*{claim*}{Claim}

\theoremstyle{definition}
\newtheorem{defn}[theorem]{Definition}
\newtheorem*{defn*}{Definition}

\newtheorem{exs}[theorem]{Examples}
\newtheorem*{exs*}{Examples}
\newtheorem{ex}[theorem]{Example}
\newtheorem*{ex*}{Example}

\newtheorem*{notn*}{Notation}

\newtheorem*{assumption*}{Assumption}
\newtheorem{rmk}[theorem]{Remark}
\newtheorem*{rmk*}{Remark}
\newtheorem{rmks}[theorem]{Remarks}
\newtheorem*{rmks*}{Remarks}

\DeclareMathOperator{\an}{an}

\DeclareMathOperator{\CC}{\mathbb{C}}
\DeclareMathOperator{\RR}{\mathbb{R}}
\DeclareMathOperator{\NN}{\mathbb{N}}

\DeclareMathOperator{\ZZ}{\mathbb{Z}}

\DeclareMathOperator{\LL}{\mathbb{L}}
\DeclareMathOperator{\A}{\mathcal{A}}
\DeclareMathOperator{\R}{\mathcal{R}}
\DeclareMathOperator{\G}{\mathcal{G}}

\DeclareMathOperator{\supp}{\textrm{supp}}
\DeclareMathOperator{\arccot}{\textrm{arccot}}

\DeclareFontFamily{U}{mathc}{}
\DeclareFontShape{U}{mathc}{m}{it}%
{<->s*[1.03] mathc10}{}
\DeclareMathAlphabet{\mathcal}{U}{mathc}{m}{it}

\newcommand{\into}{\longrightarrow}

\newcommand{\Ps}[2]{\mathbb{#1}\left[\!\left[#2\right]\!\right]}

\newcommand{\rest}[1]{\!\!\upharpoonright_{#1}}

\title{Definability of complex functions in o-minimal structures}

\author{Adele Padgett}
 \address{Kurt G\"odel Research Center, Universit\"at Wien, 1090 Wien, Austria}
 \email{adele.lee.padgett@univie.ac.at}

\author{Patrick Speissegger}
\address{Department of Mathematics and Statistics, McMaster University, Hamilton, ON, Canada}
\email{speisse@mcmaster.ca}

\date{\today}

\begin{document}

\begin{abstract}
    We prove that holomorphic continuations of functions in the classes $\mathbf{an}^*$ and $\mathcal{G}$ are definable in the o-minimal structures $\mathbb{R}_{\an^*}$ and $\mathbb{R}_{\mathcal{G}}$ respectively. More specifically, we give complex domains on which the holomorphic continuations are definable and show these domains are optimal. As an application, we describe optimal domains on which the Riemann zeta function $\zeta$ is definable in o-minimal expansions of $\mathbb{R}_{\an^*,\exp}$ and on which the Euler Gamma function $\Gamma$ is definable in o-minimal expansions of $\mathbb{R}_{\mathcal{G},\exp}$.
\end{abstract}

\maketitle

\section{Introduction}

The notion of an o-minimal structure provides a natural setting to which many classical theorems from real algebraic geometry generalize. For instance, sets in o-minimal structures can be decomposed into finitely many cells and single-variable functions are continuous outside of a finite set. The power of o-minimality is that it preserves many tameness properties from real algebraic geometry while also capturing interesting analytic objects. For example, the globally subanalytic sets form an o-minimal structure $\mathbb{R}_{\mathrm{an}}$ (Van den Dries \cite{vdDries:Tarski--SeidenbergThm}), which can be expanded to another o-minimal structure called $\mathbb{R}_{\mathrm{an},\exp}$ generated by these sets along with the graph of the real exponential function (Van den Dries and Miller \cite{vDDriesMillerRanexp}).

O-minimality can also handle holomorphic objects. Start with the semialgebraic sets, which form an o-minimal structure $\mathbb{R}_{\mathrm{alg}}$, and identify the complex numbers with $\mathbb{R}^2$. By defining addition and multiplication of complex numbers in terms of addition and multiplication of real and imaginary parts, one can work with complex algebraic varieties in $\mathbb{R}_{\mathrm{alg}}$. Similarly, one can work with the complex exponential function in $\mathbb{R}_{\an,\exp}$, as long as it is restricted to a horizontal strip, because $\exp(x+iy) = \exp(x)(\cos(y)+i\sin(y))$ can be evaluated using the real exponential function along with $\cos$ and $\sin$ restricted to an interval. (The graphs of the full sine and cosine functions are not globally subanalytic.)
As the complex exponential function is $2\pi i$-periodic, this captures its behavior anywhere in the complex plane.

One of the most striking results about o-minimal structures is a version of Chow's theorem in complex analysis. Chow's original theorem states that every analytic subvariety of complex projective space is actually an algebraic variety \cite[Theorem V]{Chow'sThm}. The version for o-minimal expansions of the real field is that if a closed analytic subset of a complex algebraic variety is \emph{definable in an o-minimal structure}, then it is a complex algebraic set.
This follows from a result of Stoll that transcendental complex analytic sets cannot satisfy a certain growth bound on the volume of their intersection with balls of increasing size \cite{stoll:VolumeBoundforAlgebraicSets}, together with the fact that o-minimal subsets of $\RR^n$ do satisfy the growth bound (Comte and Yomdin \cite[Corollary 5.2]{comte-yomdin:O-minVolumes}).
This so-called Definable Chow Theorem is a crucial piece of the Pila-Zannier strategy, which has been highly successful in using functional transcendence results to help solve problems in Diophantine geometry, like the Andr\'e-Oort conjecture (now a theorem, see Pila \cite{PilaAndreOort}, Tsimerman \cite{tsimerman:andre--ort} and Pila et al. \cite{fullAndreOort}). 

In practice, the Definable Chow Theorem is applied to analytic sets defined using holomorphic functions of arithmetic interest, like the exponential function or Klein $j$ function. Such analytic sets are definable in o-minimal structures, because large enough pieces of these holomorphic functions are definable there. As mentioned above, the restriction of the exponential function to a horizontal strip is definable in $\mathbb{R}_{\mathrm{an},\exp}$, and the $j$ function is also definable in $\RR_{\mathrm{an},\exp}$ when restricted to suitable domains.

However, in general, an o-minimal structure that defines a real-analytic function may not define any of its holomorphic continuations. For example, the o-minimal structure $\mathbb{R}_{\exp}$, which defines the real exponential function but none of the extra analytic structure of $\mathbb{R}_{\mathrm{an}}$, does not define the complex exponential function restricted to any open set in $\mathbb{C}$, see Bianconi \cite{BianconiNondefinability}. 
The question of which real-analytic functions have holomorphic continuations within a particular o-minimal structure has been studied for $\mathbb{R}_{\mathrm{an}}$ by Wilkie  \cite{WilkieHoloContinuations} and by Kaiser \cite{MR3509471}, and for $\mathbb{R}_{\mathrm{an},\exp}$ by Kaiser and Speissegger  \cite{MR3934482} and by Opris \cite{opris2024globalcomplexificationrealanalytic}. 

The main results of this paper, Theorems \ref{cgp_dfbl} and \ref{gev_dfbl} below, partially answer this question for the \hbox{o-minimal} structures $\mathbb{R}_{\mathrm{an}^*}$ and $\mathbb{R}_{\mathcal{G}}$ (see Van den Dries and Speissegger \cite{Dries:1998xr} and \cite{RealO-minGamma}, respectively). These structures are proper expansions of $\mathbb{R}_{\mathrm{an}}$, and their definitions, as well as the statements of Theorems \ref{cgp_dfbl} and \ref{gev_dfbl}, will be given in later sections. The respective expansions $\mathbb{R}_{\mathrm{an}^*,\exp}$ and $\mathbb{R}_{\mathcal{G},\exp}$ of these structures by the exponential function remain \hbox{o-minimal} \cite[Theorem B]{RealO-minGamma}, and they define, respectively, the Riemann zeta function $\zeta|_{(1,\infty)}$ on real numbers greater than one \cite[Corollary 10.11]{RealO-minGamma} and the Euler Gamma function $\Gamma|_{(0,\infty)}$ on the positive real line \cite[Example 8.1]{RealO-minGamma}.

As a consequence of Theorems \ref{cgp_dfbl} and \ref{gev_dfbl}, we obtain a precise description of the regions on which the complex $\zeta$ and $\Gamma$ functions are definable in these o-minimal structures: we denote by $\Re w$ and $\Im w$ the real and imaginary parts of $w \in \CC$, and we call a set $\Omega \subseteq \CC$ \textbf{$i$-bounded} if the set $\{\Im w:\ w \in \Omega\}$ is bounded. 

\begin{theorem}\label{zeta_dfbl}
	\begin{enumerate}
	\item For any $t \in\RR$ and $s>0$, the restriction of $\zeta$ to the set $\{z : \Re z>t, |\Im z|< s\}$ is definable in $\RR_{\an^*, \exp}$. 
	\item The restriction of $\zeta$ to any $i$-unbounded set $\Omega \subseteq \{w \in \CC:\ \Re w > 2\}$ is not definable in any o-minimal expansion of $\ \RR_{\an^*,\exp}$.
	\end{enumerate}
\end{theorem}

\begin{figure}[h]
	\centering
	\fbox{\includegraphics[width=0.45\linewidth]{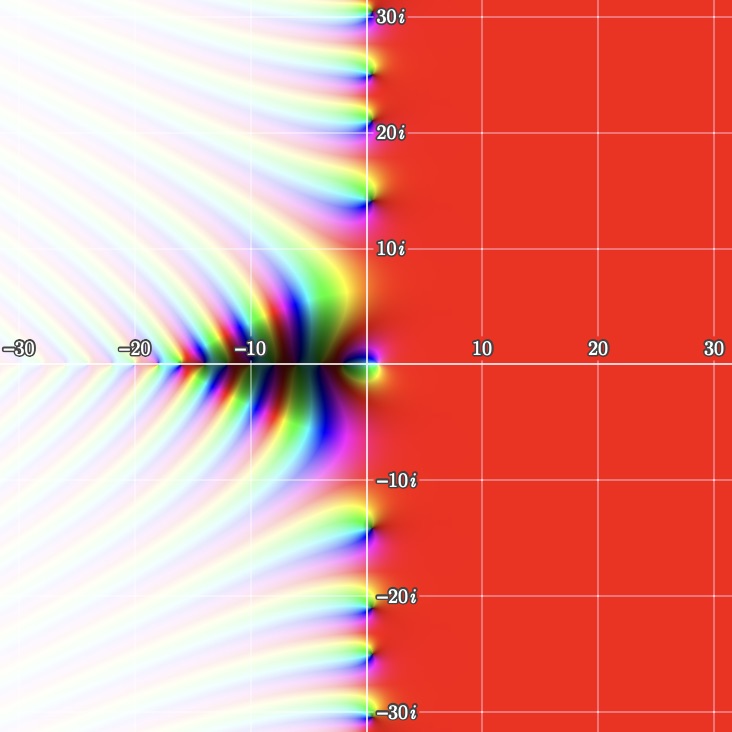}} \hspace{.02\linewidth}
	\fbox{\includegraphics[width=0.45\linewidth]{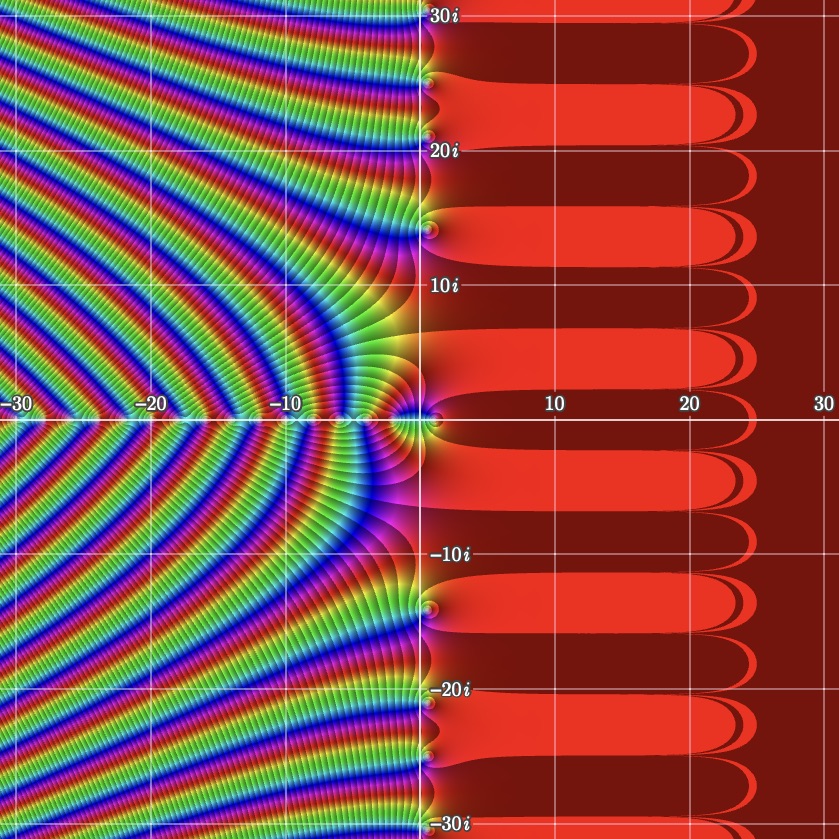}}
	\caption{Two styles of domain colorings for $\zeta(z)$ \cite{DomainColoring}.}
	\label{fig:zeta}
\end{figure}

Figure \ref{fig:zeta} provides two visualizations of the $\zeta$ function created with a freely available tool built by Li \cite{DomainColoring}. 
In both images, the color at a point $z$ represents the argument of $\zeta(z)$. In the first picture, points $z$ are very light in color if $|\zeta(z)|$ is large and close to black if $|\zeta(z)|$ is small. Much of the right half plane is vivid red (close to neither white nor black), which represents that $|\zeta(z)|$ is close to 1 in that region.
In the second picture, $|\zeta(z)|$ is represented with level curves instead of a gradient in shading. The tiny concentric loops on the negative real axis and in the critical strip represent regions where $|\zeta(z)|$ gets very small.

Similarly, we describe complex domains on which the Gamma function $\Gamma$ is definable in Theorem \ref{gamma_dfbl} below, but we defer the statement because the domains are more complicated. We also show in Proposition \ref{gamma_optimal} below that the domains obtained in Theorem \ref{gamma_dfbl} are optimal for any o-minimal expansion of the structure $\mathbb{R}_{\mathcal{G},\exp}$.  

The paper is organized as follows: in Section \ref{setup}, we make precise the notion of \emph{definable holomorphic continuation} and make some basic observations, all in an axiomatic context.  In Section  \ref{cgp_section}, we define the structure $\mathbb{R}_{\mathrm{an}^*}$ and establish Theorem \ref{cgp_dfbl}.  We also show there that this theorem is optimal in a certain sense, and we prove Theorem \ref{zeta_dfbl}. In Section \ref{multi_sec}, we define the o-minimal structure $\mathbb{R}_{\mathcal{G}}$ and establish Theorem \ref{gev_dfbl}. In Section \ref{optimal_sec} we apply Theorem \ref{gev_dfbl} to study the Stirling function, which is definable in $\RR_{\mathcal{G}}$ and used to define the $\Gamma$ function in $\mathbb{R}_{\mathcal{G},\exp}$. We also give some additional results about the optimality of the domains on which the Stirling function is definable. In Section \ref{gamma_sec}, we obtain domains of definability of the complex $\Gamma$ function, and we firm up our conclusions concerning the optimality of these domains.

\section{Definable holomorphic continuations}\label{setup}

We start by defining the notion of ``definable holomorphic continuation'' and making some related observations in the general context of a \textbf{generalized quasianalytic class} (``\textbf{GQC}'' for short) $\mathcal{A} = (\mathcal{A}_{m,n})_{m,n \in \NN}$ as defined by Rolin and Servi \cite[Section 1.2]{MR3349791}.  Their main theorem states that the expansion $\RR_{\mathcal{A}}$ of the real field by all functions in $\mathcal{A}$ defined in a neighbourhood of $[0,1]^m$ (and set equal to 0 outside this box) is o-minimal, polynomially bounded and admits quantifier elimination in its natural language augmented by symbols for division and $n$th roots.  Most currently known polynomially bounded, o-minimal expansions of the real field can be obtained from this theorem, which is why we choose this general context for this section.

In addition to the assumptions listed in \cite[Section 1.2]{MR3349791}, we make here the assumption that $\mathcal{A}$ is \textbf{analytic}, by which we mean that every germ $f \in \A_{m,n}$ has a real analytic representative on
$$I_{m,n,\rho}:= (0,\rho_1) \times \cdots \times (0,\rho_m) \times (-\rho_{m+1},\rho_{m+1}) \times \cdots \times (-\rho_{m+n}, \rho_{m+n}),$$
where $\rho \in (0,\infty)^{m+n}$ is a polyradius.
Examples of analytic GQC are:
\begin{itemize}
    \item the class \textbf{an} of all (germs defined by) convergent power series (giving rise to the o-minimal structure $\RR_{\an}$ mentioned in the introduction);
    \item the class \textbf{an}$^*$ of all (germs defined by) convergent generalized power series (giving rise to $\RR_{\an^*}$, see Section \ref{cgp_section} below);
    \item the class $\mathcal G$ of all (germs defined by) power series that are multisummable in the positive real direction (giving rise to $\RR_{\mathcal{G}}$, see Section \ref{multi_sec} below);
    \item the class $\mathcal G^*$ of all (germs defined by) generalized power series that are multisummable in the positive real direction (giving rise to $\RR_{\mathcal{G}^*}$, see Rolin et al. \cite{Rolin:2022wa});
    \item and the class $\mathcal Q$ of all (germs defined by) generalized power series that are almost regular (giving rise to $\RR_{\mathcal{Q}}$, see Kaiser et al. \cite{MR2572245}). 
\end{itemize}   
Note that the first four examples are reducts of $\RR_{\mathcal{G}^*}$, while the last one is believed to be distinct from the former.

Every real analytic function has a holomorphic continuation on some complex domain.  The germs in $\A$ often have a branch point at the origin, so we will also be considering holomorphic continuations on the Riemann surface of the logarithm
$$\LL := (0,\infty) \times \RR,$$
where, for $z = (r,\theta) \in \LL$, we call $|z|:= r$ the \textbf{modulus} of $z$ and $\arg z:= \theta$ the \textbf{argument} of $z$.  For the purpose of continuation of germs in $\A$ on $\LL$, we identify the subset $(0,\infty) \times \{0\}$ with the real half-line $(0,\infty)$.  The bijection $L:\LL \into \CC$ defined by $L(r,\theta):= \log r + i\theta$ equips $\LL$ with a structure of a holomorphic manifold, and we denote by $E:\CC \into \LL$ its compositional inverse.  Thus, if $\Omega \subseteq \LL$ is a domain and $\varphi:\Omega \into \CC$ is a function, then $\varphi$ is holomorphic if and only if $\varphi \circ E:L(\Omega) \into \CC$ is holomorphic.  Note that $\LL$ is definable in the real field, and $L$ is definable any o-minimal expansion of the real field in which $\log :(0,\infty)\to \RR$ is definable.

Finally, we let $\bar\LL:= \LL \cup \{0\}$ and set $|0|:=0$ and $\arg 0:= 0$, and we extend the topology on $\LL$ to $\bar\LL$ by taking the \textbf{$\log$-disks} $D_{\bar\LL}(R):= \{z \in \bar\LL: |z| < R\}$, for $R>0$, as the basic open neighbourhoods of $0$.

\begin{defn}\label{dfbl_reps}
    Let $\R$ be an o-minimal expansion of the real field and $\varphi:\Omega \longrightarrow \CC$ be a function, where $\Omega \subseteq \bar\LL^m \times \CC^n$.  We set $$\Omega^{\RR} := \{(r,\theta,x,y) \in (0,\infty)^m \times \RR^m \times \RR^{2n}:\ ((r,\theta),x + iy) \in \Omega\},$$ where $(r,\theta) := ((r_1,\theta_1), \dots, (r_m,\theta_m)) \in \LL^m$, and we define $\varphi^{\RR}:\Omega^{\RR} \longrightarrow \CC$ by $$\varphi^{\RR}(r,\theta,x,y) := \varphi((r,\theta),x+iy).$$ We say that $\varphi$ is \textbf{definable in $\R$} if both the real part $\Re \varphi^{\RR}$ and the imaginary part $\Im \varphi^{\RR}$ of $\varphi^{\RR}$ are definable in $\R$.  
\end{defn}

\begin{exs} \label{dfbl_exs}
	\begin{enumerate}
		\item If $\Omega \subseteq \CC^n$ is bounded and definable in $\RR_{\an}$ and $\varphi$ is the restriction to $\Omega$ of a meromorphic function on an open set containing $\Omega$, then $\varphi$ is definable in $\RR_{\an}$.
		\item The function $e^{i\theta}: (-\pi,\pi) \into \CC$ is definable in an o-minimal expansion $\R$ of the real field if and only if both $\sin$ and $\cos$ restricted to $(-\pi,\pi)$ are definable in $\R$.
	\end{enumerate}
\end{exs}

To talk about continuations of real germs at the origin on $\bar\LL$, we need the covering map $\Pi:\LL \into \CC^\times$ given by 
$$\Pi(r,\theta):= re^{i\theta},$$
which is not definable in any o-minimal expansion of the real field.  We extend $\Pi$ to $\bar\LL$ by setting $\Pi(0):= 0$, and we let $\Pi_0$ be the restriction of $\Pi$ to $\{z \in \bar\LL:\ |\arg z| < \pi\}$.  Note that $\Pi_0$ is injective, its image is $\CC \setminus (0,-\infty)$, and it is definable in an o-minimal expansion $\R$ of the real field if and only if $e^{i\theta}: (-\pi,\pi) \into \CC$ is definable in $\R$.  Below, we also write $\Pi$ and $\Pi_0$ for the componentwise application of $\Pi$ and $\Pi_0$ from $\bar\LL^n$ to $\CC^n$, respectively.

\begin{rmk}\label{dfbl_rmk}
	Let $\R$ be an o-minimal expansion of the real field, and assume that $e^{i\theta}: (-\pi,\pi) \into \CC$ is definable in $\R$.  Let $\varphi:\Omega \longrightarrow \CC$ be a function, where $\Omega \subseteq (\CC \setminus (0,-\infty))^n$.  Then $\varphi$ is definable in $\R$ if and only if $\varphi \circ \Pi_0:\Pi_0^{-1}(\Omega) \longrightarrow \CC$ is definable in $\R$.
\end{rmk}

Returning to the GQC $\mathcal A$: given a germ $f \in \mathcal A_{m,n}$ and a function $\varphi:\Omega \into \CC$ with $\Omega \subseteq \bar\LL^m \times \CC^n$, we call $\varphi$ a \textbf{representative} of $f$, if $\varphi$ is holomorphic on the interior of $\Omega$ and $$f_{\varphi}:= \varphi(\Pi_0^{-1}(x),y):I_{m,n,\rho} \into \RR$$ is a real representative of $f$, for some polyradius $\rho \in (0,\infty)^{m+n}$. In this situation, we also refer to $\varphi$ as a \textbf{continuation} of $f_{\varphi}$.

Definability of continuations of all functions definable in $\RR_{\an}$ (not just of the primitives) was established in \cite{MR3509471}, while holomorphic continuations of all definable univariate germs in $\RR_{\an,\exp}$ were studied in \cite{MR3934482}.  This paper can be considered as a first step in this direction for $\A = $ \textbf{an}$^*$ and $\A = \G$.  

\section{Convergent generalized power series} \label{cgp_section}

In this section, we consider the case $\A = $ \textbf{an}$^*$.  For each $m,n \in \NN$, the ring \textbf{an}$^*_{m,n}$ is the set of all germs at the origin of real functions obtained as follows:  let $X = (X_1, \dots, X_m)$ and $Y = (Y_1, \dots, Y_n)$ be tuples of indeterminates, $\alpha = (\alpha_1, \dots, \alpha_m)$ range over $[0,\infty)^m$ and $\beta = (\beta_1, \dots, \beta_n)$ range over $\NN^n$.  Let $F(X,Y) = \sum_{\alpha,\beta} a_{\alpha,\beta} X^\alpha Y^\beta$ be a \textbf{mixed generalized power series} as defined in \cite{Dries:1998xr}; that is, there are well-ordered sets $A_1, \dots, A_m \subseteq [0,\infty)$ such that the \textbf{support} of $F$,
$$\supp(F):= \{(\alpha,\beta):\ a_{\alpha,\beta} \ne 0\}$$
is a subset of $A_1 \times \cdots \times A_m \times \NN^n$.  We refer to $X$ and $Y$ as the \textbf{generalized} and \textbf{standard} indeterminates of $F$, respectively.
We say that $F$  has \textbf{polyradius of convergence at least} $(r,s)$, where $r \in (0,\infty)^m$ and $s \in (0,\infty)^n$ are polyradii, if $$\|F\|_{r,s} := \sum |a_{\alpha,\beta}| r^\alpha s^\beta < \infty.$$ 
In this situation, setting 
$$D_{\LL}(r):= \{z \in \LL^m:\ |z_i| < r_i \text{ for each } i\},\quad D_{\bar\LL}(r):= \{z \in \bar\LL^m:\ |z_i| < r_i \text{ for each } i\}$$
and
$$D(s):= \{w \in \CC^n:\ |w_i| < s_i \text{ for each } i\},$$
the series $F$ defines a continuous function $F_{r,s}:D_{\bar\LL}(r) \times D(s) \into \CC$, given by $$F_{r,s}(x,y):= \sum_{\alpha,\beta} a_{\alpha,\beta} |x|^\alpha e^{i\alpha\arg x} y^\beta,$$ whose restriction to $D_{\LL}(r) \times D(s)$ is holomorphic.  For $\rho > 0$, we define the \textbf{polysector}
$$S_{\bar\LL}(r,\rho):= \{z \in \bar\LL^m:\ |z_i| < r_i \text{ and } |\arg z_i| < \rho \text{ for each } i\},$$
and we obtain:

\begin{theorem}  \label{cgp_dfbl}
	Let $F(X,Y) = \sum_{\alpha,\beta} a_{\alpha,\beta} X^\alpha Y^\beta$ be a mixed generalized power series with polyradius of convergence at least $(r,s)$, and let $\rho >0$.  Then the restriction of $F_{r,s}$ to $S_{\bar\LL}(r,\rho) \times D(s)$ is definable in $\RR_{\an^*}$.
\end{theorem}

Before proving Theorem \ref{cgp_dfbl}, we show how to apply it to the Riemann $\zeta$ function.

\begin{proof}[Proof of Theorem \ref{zeta_dfbl}(1)]
    Consider the generalized power series
$$F^{\zeta}(X) := \sum_{n \ge 1} X^{\log n},$$
which has radius of convergence $1/e$, that is, radius of convergence at least $r$ for every $r \in (0,1/e)$.  Hence $F^{\zeta}$ has a representative $f^{\zeta}:D_{\bar\LL}(1/e) \into \CC$ and, by definition, we have 
\begin{equation} \label{zeta_def}
	\zeta(w) = f^{\zeta}(E(-w)), \quad\text{for } w \in \CC \text{ with } \Re w > 1.
\end{equation} 
We conclude using Theorem \ref{cgp_dfbl} and Example \ref{dfbl_exs}(1).
\end{proof}

The proof of Theorem \ref{cgp_dfbl} is based on the following lemma:

\begin{lemma} \label{cgp_pol_dfbl}
	Let $F(X,Y) = \sum_{\alpha,\beta} a_{\alpha,\beta} X^\alpha Y^\beta$ be a mixed generalized power series with polyradius of convergence at least $(r,s)$.
	Then there are mixed generalized power series $G$ and $H$, with generalized indeterminates $X$ and standard indeterminates $(U,Y,V)$, where $U = (U_1, \dots, U_m)$ and $V = (V_1, \dots, V_n)$, such that,
	for any polyradius $$\tau = \left(r',\rho,\frac s2, \frac s2\right) \in (0,\infty)^{2m+2n} \text{ satisfying } r'e^\rho < r,$$ we have
	$$\Re F_{r,s}\left((x,u),y+iv\right) = G_{\tau}(x,u,y,v) \quad \text{and} \quad \Im F_{r,s}\left((x,u),y+iv\right) = H_{\tau}(x,u,y,v)$$
	for all $(x,(u,y,v)) \in I_{m,m+2n,\tau}$.  In particular, the restriction of $F_{r,s}$ to the set $S_{\bar\LL}(r',\rho) \times D(s/2)$ is definable in $\RR_{\an^*}$.
\end{lemma}

\begin{proof} 
	Let $r',\rho \in (0,\infty)^m$ be polyradii satisfying $r'e^\rho < r$, and let $(x,(u,y,v)) \in I_{m,m+2n,\tau}$.  Then, by definition of $F_{r,s}$,
	\begin{align*}
	    F_{r,s}&((x,u),y+iv) = \sum_{\alpha,\beta} a_{\alpha,\beta} x^\alpha e^{iu\alpha}(y+iv)^\beta \\
    &= \sum_{\alpha,\beta} a_{\alpha,\beta} x^\alpha (\cos(u_1\alpha_1) + i\sin(u_1\alpha_1)) \cdots (\cos(u_m\alpha_m) + i\sin(u_m\alpha_m))(y+iv)^\beta.
	\end{align*}
	Let $C(T)$ be the Taylor series of $\cos t$, and let $S(T)$ be the Taylor series of $\sin t$ (both at $t=0$).  Replacing $\cos$ and $\sin$ in the above by their Taylor series and collecting real and imaginary parts, we obtain two mixed generalized series $G(X,U,Y,V)$ and $H(X,U,Y,V)$ as stated in the lemma.  Since $\|C\|_t + \|S\|_t = e^t$ and $\|(Y+iV)^n\|_{t,t} = (2t)^n$, for $t>0$ and $n \in \NN$, we get  
	\begin{align*}
	    \|G\|_{\tau} &\le  \sum_{\alpha,\beta} |a_{\alpha,\beta}| (r')^\alpha (\|C\|_{\rho_1\alpha_1} + \|S\|_{\rho_1\alpha_1}) \cdots (\|C\|_{\rho_m\alpha_m} + \|S\|_{\rho_m\alpha_m})s^\beta \\
	    &= \sum_{\alpha,\beta} |a_{\alpha,\beta}| (r')^\alpha e^{\rho\alpha} s^\beta \\
	    &= \sum_{\alpha\beta} |a_{\alpha,\beta}| (r'e^\rho)^\alpha s^\beta.
	\end{align*}
	Since the latter sum converges whenever $r'e^\rho < r$, it follows that $\Re F_{r,s}((x,u),y+iv) = G_{\tau}(x,u,y,v)$.  The same argument shows that $\Im F_{r,s}((x,u),y+iv) = H_{\tau}(x,u,y,v)$.  The definability of $G_\tau$ and $H_\tau$ in $\RR_{\an^*}$ follows from \cite[Lemma 7.4]{Dries:1998xr}.
\end{proof}

\begin{proof}[Proof of Theorem \ref{cgp_dfbl}]
	Let $r' \in \left(0,\frac{r}{e^\rho}\right)$, and set  $$\Omega:= (S_{\bar\LL}(r,\rho) \times D(s))\setminus (S_{\bar\LL}(r',\rho) \times D(s/2)).$$  By Lemma \ref{cgp_pol_dfbl}, the restriction of $F_{r,s}$ to $S_{\bar\LL}(r',\rho) \times D(s/2)$ is definable in $\RR_{\an^*}$, so it suffices to show that the restriction to $\Omega$ is definable.  
	
	Now note that the restriction of $\log$ to any closed and bounded interval contained in $(0,\infty)$ is definable in $\RR_{\an}$.  Therefore, $L$ restricted to $S_{\LL}(t,\sigma) \setminus S_{\LL}(t',\sigma)$ is definable for any $t > r$, any $t' \in (0,r')$ and any $\sigma > \rho$.  Since $\|F\|_{r,s} < \infty$, there exists $t>r$ such that $\|F\|_{t,s} < \infty$ as well; in particular, the function $F_{t,s}$ is a holomorphic continuation of the restriction of $F_{r,s}$ to $\bar\Omega$ on an open neighbourhood of $\bar\Omega$.  Therefore, the function $\varphi:\bar\Delta \into \CC$, defined by $$\bar\Delta:= \left\{(L(x),y):\ (x,y) \in \bar\Omega\right\} \subseteq \CC^{m+n} \quad\text{and}\quad \varphi(z,y):= F_{r,s}((E(z),y))$$ is holomorphic on the compact definable set $\bar\Delta$, hence definable in $\RR_{\an}$ by Example \ref{dfbl_exs}(1).
\end{proof}

The next proposition shows that Theorem \ref{cgp_dfbl} is optimal in the following sense: we call $\Omega \subseteq \bar\LL$ \textbf{argument-bounded} if the set $$A_\Omega := \{\arg z:\ z \in \Omega\} \subseteq \RR$$ is bounded.  Let $F(X) = \sum_{\alpha} a_{\alpha}X^{\alpha}$ be a nonconstant generalized power series in one indeterminate $X$, and assume that $F$ has radius of convergence at least $r>0$.  Then $$F(X) - F(0) = a_{\alpha_0}X^{\alpha_0} (1 + G(X)),$$  where $\alpha_0 := \min\supp(F(X)-F(0)) > 0$ and $$G(X) := \sum_{\alpha>\alpha_0} \frac{a_{\alpha}}{a_{\alpha_0}}X^{\alpha-\alpha_0}.$$  Since $G$ also has radius of convergence at least $r$, we have $\displaystyle \lim_{\rho \to 0} \|G\|_{\rho} = 0$, so we set $$\rho:= \sup\{t \in (0,r):\ \|G\|_t < 1.$$

\begin{prop} \label{cgp_ndfbl}
	The restriction of $F_r$ to $\Omega$ is not definable in any o-minimal expansion of $\RR_{\an^*}$, for any argument-unbounded $\Omega \subseteq D_{\bar\LL}(\rho)$.
\end{prop}

\begin{proof}
	For $x = (|x|,\arg x) \in D_{\bar\LL}(\rho)$, we have
	\begin{equation} \label{curve_eq}
	F_{r}(x) - F_{r}(0)
	= a_{\alpha_0}|x|^{\alpha_0} e^{i\alpha_0 \arg x} \left(1+G_{r}(x)\right),
	\end{equation}
	with $|G_{r}(x)| < 1$.
	Assume for a contradiction that there is an argument-unbounded set $\Omega \subseteq D_{\bar\LL}(\rho)$ such that the restriction $F_{\Omega}$ of $F_{r}$ to $\Omega$ is definable in some o-minimal expansion $\R$ of $\RR_{\an^*}$.  Then, by o-minimality, $A_{\Omega}$ contains an interval $(a,\infty)$ or $(-\infty,a)$, for some $a \in \RR$; we assume here the former, the latter being handled similarly.

	By definable curve selection, there is a definable (in $\R$) curve $\gamma:(0,\infty) \into \Omega$ such that $\arg\gamma(t) = t$ for all $t > a$; in particular, we have $|G_r(\gamma(t))| < 1$ for all $t>a$.  By the Monotonicity Theorem, after increasing $a$ if necessary, we may assume that $\gamma$ is continuous.  

	Since the power function $t \mapsto t^{\alpha_0}:(0,\infty) \into \RR$ is definable in $\RR_{\an^*}$, it follows from the definability of $F_\Omega$ in $\R$ that the curve $\delta:(a,\infty) \into \CC$ defined by $$\delta(t):= \frac{F(\gamma(t))-F(0)}{a_{\alpha_0}|\gamma(t)|^{\alpha_0}}$$
	is definable in $\R$.
	However, since $e^{i\alpha_0 \theta}$ is periodic and of modulus 1, and since $|G_r(\gamma(t))| < 1$ for all sufficiently large $t$, the continuous curve $\epsilon(t):= e^{i\alpha_0 \arg\gamma(t)} \big(1+G_r(\gamma(t))\big):(a,\infty) \into \CC$ intersects the real axis in infinitely many connected components.  By Equation \eqref{curve_eq}, we have $\epsilon = \delta$, which contradicts the definability of $\delta$.
\end{proof}

\begin{proof}[Proof of Theorem \ref{zeta_dfbl}(2)]
	We have $$F^\zeta(X) - 1 = G(X):= \sum_{n>1} X^{\log n}.$$  We get from Calculus that $\left\|G\right\|_{t} < 1$ for all $t < e^{-2}$.  So part (2) of Theorem \ref{zeta_dfbl} follows from Proposition \ref{cgp_ndfbl} and Equation \eqref{zeta_def}.
\end{proof}

\section{Multisummable series} \label{multi_sec}

We now move on to the case $\A = \G$; we first recall some notation and the definitions of generalized sectors and multisummable functions from \cite{RealO-minGamma}.  The only difference is that here the generalized variables range over the Riemann surface of $\log$.

For $(k_1,\dots,k_m) \in [0,\infty)^m$ and $z = ((|z_1|,\arg z_1),\dots,(|z_m|,\arg z_m)) \in \bar\LL^m$, we put 
    \begin{align*}
        k \cdot |\arg z| &:= k_1|\arg z_1| + \cdots + k_m|\arg z_m| \\
        z^k &:= \left(|z_1|^{k_1} \cdots |z_m|^{k_m}, k_1\arg z_1 + \cdots + k_m\arg z_m\right)\\
        |z| &:= \sup\{|z_i| : i = 1,\dots,m\} 
    \end{align*}
    For a polyradius $R=(R_1,\dots,R_m) \in (0,\infty)^m$, we put
    \begin{align*}
        [0,R) := [0,R_1) \times \cdots \times [0,R_m) \subset \RR^m.
    \end{align*}
    For $R,\widetilde{R} \in (0,\infty)^m$ we write $R \le \widetilde{R}$ if $R_i \le \widetilde{R}_i$ for each $i$, and $R < \widetilde{R}$ if $R_i < \widetilde{R}_i$ for each $i$.
	If $z \in \CC^m$ and $f:\CC \into \CC$ is a function, we will write $$f(z):= (f(z_1),\dots,f(z_n));$$ similarly for $f:\bar\LL \into \CC$.
    If $a,b \in \bar\LL^m$, we denote by $ab$ the coordinatewise product $(a_1b_1,\dots,a_mb_m)$.

    Let $R \in(0,\infty)^m$ be a polyradius, $\phi \in (0,\pi)$, and $k \in [0,\infty)^m$.
    The \textbf{generalized sector} is the set
    \begin{align*}
        S_{\bar\LL}(k,R,\phi) &:= \{z \in D_{\bar\LL}(R) : k \cdot|\arg z| < \phi\}.
    \end{align*}
    Correspondingly, for $p \in \NN$, we set 
    \begin{align*}
        D_{\bar\LL}(k,R,p) &:= \left\{z \in D_{\bar\LL}(R) : |z^k| < \frac{R^k}{p+1}\right\} \\
        S_{\bar\LL}(k,R,\phi,p) &:= S_{\bar\LL}(k,R,\phi) \cup D_{\bar\LL}(k,R,p).
    \end{align*}
    For a nonempty finite subset $K \subset [0,\infty)^m$, we set
    \begin{align*}
        S_{\bar\LL}(K,R,\phi) &:= \bigcap_{k \in K} S_{\bar\LL}(k,R,\phi) \\
        S_{\bar\LL}(K,R,\phi,p) &:= \bigcap_{k \in K} S_{\bar\LL}(k,R,\phi,p).
    \end{align*}

For the next definition, we also fix $r>1$.  To lighten notation, we set $\tau:=(K,R,r,\phi)$, and we write $S(\tau) := S_{\bar\LL}(K,R,\phi)$ and $S_p(\tau) := S_{\bar\LL}(K,R,\phi,p)$; if $\tau$ is clear from context, we shall also simply write $S$ and $S_p$, respectively.

\begin{defn}
    For each $p \in \NN$ let $f_p : S_p \to \CC$ be a function which is holomorphic on $S_p\setminus \{0\}$ and satisfies the following conditions:
    \begin{enumerate}
        \item The function $f_p$ is bounded and 
        \[
            \sum_{p \in \NN} \|f_p\|_{S_p}r^p < \infty
        \]
        where $\|f\|_U := \sup_{z \in U} |f(z)| \in [0,\infty]$ for a function $f : U \to \CC$.
        
        \item There is a power series $F_p(X) = \sum_{\alpha}a_{\alpha} X^{\alpha}$ with $\supp(F)\subset \NN^m$ and polyradius of convergence at least $\frac{R}{\sqrt[k]{p+1}}$, such that $f_p(z)=(F_p)_{\frac{R}{\sqrt[k]{p+1}}}(z)$ for all $z \in D_{\bar{\LL}}\left(\frac{R}{\sqrt[k]{p+1}}\right)$.
    \end{enumerate}
    By condition (1), $\sum_{p \in \NN}f_p$ converges uniformly on $S \setminus\{0\}$ to a holomorphic function $f : S\setminus\{0\} \to \CC$. 
    Both conditions together imply that $\sum_{n\in\NN} F_p$ converges (in the sense of formal power series) to a series $F$, generally with polyradius of convergence zero.
    By condition (2) each $f_p$ is holomorphic at $0 \in \bar{\LL}$, so again both conditions imply $f$ extends to $0$ and is $C^{\infty}$ there with Taylor series $F$.
    We denote this state of affairs by \[f =_{\tau}\sum_{p \in \NN}f_p.\]
    
    Let $\mathcal{G}_{\tau}$ be the set of all functions $f : S \to \CC$ such that $f =_{\tau} \sum_{p \in \NN} f_p$ for some such sequence $(f_p)_{p \in \NN}$.
    For $f \in \mathcal{G}_{\tau}$, we put $$\|f\|_{\tau} := \inf\left\{\sum_{p \in \NN}\|f_p\|_{S_p}r^p:\ f =_{\tau} \sum_{p \in \NN} f_p\right\}.$$   
\end{defn}

\begin{ex} \label{sterling_ex}
	Recall that
	\[
	\Gamma(z) = \sqrt{2\pi}z^{z-\frac{1}{2}}e^{-z}e^{\varphi(z)}
	\]
	for $z \in \CC\setminus (-\infty,0]$, where $\varphi(z)$ is the Stirling function.
	Let $\psi(z) = \varphi \left(\frac{1}{z}\right)$; by Sauzin \cite[Theorem 5.41]{Sauzin}, $\psi$ is $C^\infty$ at 0.  Its Taylor series at 0, denoted here by $\hat\psi$, is 1-summable in every direction $d \in (-\frac\pi2, \frac\pi2)$, and $\psi$ is the Borel sum of $\hat\psi$.  
	Given $R>0$ and $\alpha \in (\frac\pi2,\pi)$ this implies, by Tougeron \cite[Prop. 2.9]{Tougeron:1994fk}, that there exists $r>1$ such that the restriction of $\psi \circ \Pi$ to $S_{\bar\LL}(\tau)$ belongs to $\G_\tau$, where $\tau = (\{1\},R,r,\alpha)$.
\end{ex}

The corresponding generalized quasianalytic class $\G$ is then defined as follows: recall from \cite[Section 3]{RealO-minGamma} that a series $F = \sum_{\alpha \in \NN^m} F_\alpha Y^\alpha \in \G_{\tau}\left[\!\left[Y\right]\!\right]$ is \textbf{mixed multisummable} (or \textbf{mixed} for short) \textbf{with polyradius of convergence at least $\rho$}, if $$\|F\|_{\tau,\rho}:= \sum_{\alpha \in \NN^m} \|F_\alpha\|_{\tau} (\rho)^\alpha < \infty.$$  Such a series $F$ defines a holomorphic function $F_{\tau,\rho}:S(\tau) \times D(\rho) \longrightarrow \RR$, given by $$F_{\tau,\rho}(u,w):= \sum_{\alpha \in \NN^m} F_\alpha(u) w^\alpha.$$
The ring $\G_{m,n}$ is then the set of all germs at the origin of functions $f:I_{m,n,(R,\rho)} \into \RR$, for which there exist $\tau = (K,R,r,\phi)$, $\rho \in (0,\infty)^n$ and a mixed series $F = \sum_{\beta} F_\beta Y^\beta \in \Ps{\G_\tau}{Y}$ with radius of convergence at least $\rho$ such that $f$ is the restriction of $F_{\tau,\rho}$ to $I_{m,n,(R,\rho)}$.  

\begin{ex} \label{multisummable_germs}
	By \cite[Props. 1.7(2) and 2.9]{Tougeron:1994fk}, the set $\G_{1,0}$ is exactly the set of all real germs at $0^+$ of Borel sums of power series that are multisummable in the positive real direction, as defined in Balser \cite[Section 10.2]{Balser:2000fk}.
\end{ex}

\begin{defn}\label{predicates}
    For $\rho \in (0,\infty)^m$, let $\mathcal{G}(m,\rho)$ be the set of functions $f : \RR^m \to \RR$ with the following property: there exist a tuple $\tau = (K,R,r,\phi)$ with $R > \rho$ and $\phi \in \left(\frac{\pi}{2},\pi\right)$, and a function $g \in \mathcal{G}_{\tau}$, such that $$f(x) = \begin{cases} g(x) &
    \text{if } x \in [0,\rho], \\ 0 &\text{otherwise.}\end{cases}$$
\end{defn}

For each $m$, the set $\mathcal{G}(m,\rho)$ is a ring that contains all real constant functions on $[0,\rho]$ and is closed under taking partial derivatives $\partial/\partial x_i$ (see \cite[Section 2]{RealO-minGamma}); in particular, each function $f \in \G(m,\rho)$ is of class $C^\infty$ on $[0,\rho]$.  It is shown in \cite[Theorem A]{RealO-minGamma} that the structure 
\[
    \RR_{\mathcal{G}} = \left(\RR,<,+,-,\cdot,0,1,\{f\}_{f \in \G(m,1), m \in \NN}\right)
\]
is model complete, o-minimal and polynomially bounded.  

\begin{theorem}\label{gev_dfbl}
	Let $K \subseteq [0,\infty)^m$ be nonempty and finite, $R \in (0,\infty)^m$, $r>1$ and $\phi \in \left(\frac{\pi}{2},\pi\right)$, and set $\tau := (K,R,r,\phi)$.  Set also $M := \max\{k_1 + \dots + k_m:\ k \in K\}$, let $\mu \in \left(0,\frac{\phi-\pi/2}{M}\right)$ and $\rho \in (0,R)$, and set $\tau':=(K,\rho,r,\mu)$. 
	Then for $f \in \G_\tau$, the restriction of $f$ to $S(\tau')$ is definable in $\RR_{\G}$.  
\end{theorem}

\begin{ex} \label{sterling_dfbl}
	In the case of the Stirling function $\varphi$ of Example \ref{sterling_ex}, for $\psi$ we have $K = \{1\}$ and $M=1$.  Thus, for every $R>0$ and $\mu \in \left(0, \frac\pi2\right)$, the restriction of $\psi \circ \Pi$ to $S_{\bar\LL}(R,\mu) = S_{\bar\LL}(1,R,\mu)$ is definable in $\RR_{\G}$.
\end{ex}

The proof of Theorem \ref{gev_dfbl} needs a bit of preparation:  let $\tau$, $M$ and $\tau'$ be as in the theorem, and 
let $\nu \in \left(\mu,\frac{\phi-\pi/2}{M}\right)$. 
For $j=1,\dots,m$, we set $\displaystyle R_j' := \frac{R_j}{e^{\nu}}$ and write $R' = (R_1',\dots,R_m')$, $\bar\mu = (\overbrace{\mu,\dots,\mu}^{m\text{ times}})$ and $\bar{\nu} = (\overbrace{\nu,\dots,\nu}^{m\text{ times}})$.
We set $\delta:= \phi- M\mu$, $\epsilon:= \phi-M\nu$ and $$\sigma := (K,R',r,\epsilon);$$  then $\frac\pi2 < \epsilon < \delta < \phi$.

\begin{lemma} \label{domain_lemma}
    Let $z \in S_p(\sigma)$ and $w \in D(\bar\nu)$.  Then $z E(iw) \in S_p(\tau)$.
\end{lemma}

\begin{proof}
	It suffices to prove the lemma for $K = \{k\}$ a singleton.
	Write $z \in \bar\LL^m$ as $z =  \left(\left(|z_1|,\arg(z_1)\right),\dots,\left(|z_m|,\arg(z_m)\right)\right)$ and split the vector $w \in D$ into its real and imaginary parts: $w = u + iv$ with $u,v \in \left(-\nu,\nu\right)^m$.
	Then, given $z \in \bar\LL^m$ and $w \in \CC^m$, we find $x,y \in \RR^m$ such that $x E(iy) = z E(iw)$ as follows:
	\begin{align*}
	    z E(iw) 
        &= \big(z_1E(iw_1), \cdots, z_mE(iw_m)\big) \\
        &= \left(\left(|z_1|e^{v_1}, \arg z_1 + u_1\right), \dots, \left(|z_m|e^{v_m}, \arg z_m+u_m\right)\right) .
	\end{align*}
	So we take $x_j := |z_j| e^{v_j}$ for each $j$ and $ y := \arg z+u$.

	First suppose $z \in S(\sigma)$. 
	Then $\displaystyle |z_j| \le R_j' = \frac{R_j}{e^{\nu}}$ for each $j$.
	Since $w \in D(\bar\nu)$, we have $|u_j| < \nu$ and $|v_j| < \nu$ for each $j$.
	So $|x_j| = |z_j|e^{v_j} \le R_j$.
	By hypothesis, we have $k\cdot |\arg z|<\epsilon$.
	Therefore,
	\begin{align*}
    k \cdot |\arg(zE(iw))| &= k \cdot|y| \\
        &= k \cdot |\arg z+u| \\
        &< k \cdot |\arg z| + k \cdot |u| \\
        &< \epsilon + M\nu = \phi;
	\end{align*}
	hence $zE(iw) \in S(\tau)$ in this case.

	Now suppose $z \in D_{\bar\LL}(k,R',p)$.
	Then $|z|^k = \left|z^k\right| < \frac{(R')^k}{p+1} = \frac{R^k}{(p+1)e^{M\nu}}$.
	Therefore,
	\begin{align*}
	    \left|(zE(iw))^k\right| 
	        &= x^k \\
	        &\le \left|z\right|^k(e^{v})^k\\
	        &< \frac{R^k}{p+1}\frac{e^{k \cdot v}}{e^{M\nu}}\\
	        &< \frac{R^k}{p+1}.
	\end{align*}
	So $zE(iw) \in D_{\bar\LL}(k,R,p)$ in this case. 
\end{proof}

We now fix a sequence $(f_p)_{n \in \NN}$ such that $f =_\tau \sum_p f_p$.  By Lemma \ref{domain_lemma}, there are holomorphic functions $g_p,\hat{g_p}:S_p(\sigma) \times D(\bar\nu) \into \CC$ defined by $$g_p(z,w) := f_p(z E(iw)) \quad\text{ and }\quad \hat{g_p}(z,w) := \overline{f_p(\bar z E(-i\bar w))}.$$
Then the two functions $f_p^r,f_p^i:S_p(\sigma) \times D(\bar\nu) \into \CC$ defined by
\begin{equation*}
    f_p^r(z,w) := \frac{g_p(z,w) + \hat{g_p}(z,-w)}{2} \quad\text{ and }\quad
    f_p^i(z,w) := \frac{g_p(z,w) - \hat{g_p}(z,-w)}{2}
\end{equation*}
satisfy the following: for all real $(x,\theta) \in S_p(\sigma) \times D(\bar\nu)$, we have 
\begin{equation} \label{real_part}
    f_p^r(x,\theta) 
        = \frac{g_p(x,\theta) + \hat{g_p}(x,-\theta)}{2} 
        = \frac{f_p(xE(i\theta)) + \overline{f_p(xE(i\theta))}}{2} 
        = \Re f_p(xE(i\theta))
\end{equation}
and similarly 
\begin{equation} \label{im_part}
	f_p^i(x,\theta) = \Im f_p(xE(i\theta)).
\end{equation}

\begin{lemma} \label{convergence_lemma}
	The sums $\sum_{p}f_p^r$ and $\sum_{p}f_p^i$ converge to holomorphic functions $f^r$ and $f^i$ on $S(\sigma) \times D(\bar\nu)$, respectively.  
\end{lemma}

\begin{proof}
First, observe that for all $(z,w) \in S_p(\sigma) \times D(\bar\nu)$, we have
\begin{align*}
    |f_p^r(z,w)| 
        &= \left|\frac{g_p(z,w) + \hat{g_p}(z,-w)}{2}\right| \\
        &= \frac{\left|f_p(zE(iw)) + \overline{f_p(\bar zE(i\bar w))}\right|}{2} \\
        &\le \left|f_p(zE(iw))\right| \\
        &\le \|f_p\|_{S_p(\tau)}
\end{align*}
and similarly, $|f_p^i(z,w)| \le \|f_p\|_{S_p(\tau)}$.
Recall that $r \in (1,\infty)$ is such that $\displaystyle \sum_{p \in \NN}\|f_p\|_{S_p(\tau)}\cdot r^p < \infty$.
So
\[
    \sum_{p \in \NN}\|f_p^r\|_{S_p(\sigma) \times D(\bar\nu)}\cdot r^p
        \le \sum_{p \in \NN}\|f_p\|_{S_p(\tau)}\cdot r^p < \infty,
\]
and similarly for $f^i$, so the lemma follows.
\end{proof}

\begin{lemma} \label{mixed_dfbl}
	There are mixed series $F^r, F^i \in \G_{\sigma}\left[\!\left[Y\right]\!\right]$ with polyradius of convergence at least $\bar\mu$ such that the restrictions of $f^r $ and $f^i$ to $S(\sigma) \times D(\bar\mu)$ agree with $F^r_{\sigma,\bar\mu}$ and $F^i_{\sigma,\bar\mu}$, respectively.
\end{lemma}

\begin{proof}
	We give the proof for $f^r$; the proof for $f^i$ is similar.  To simplify notation, we omit the superscript $r$ below.
	Fix $p \in \NN$; by Taylor's Theorem we have, for each $(z,w) \in S_p(\sigma) \times D(\bar\nu)$, that 
	$$f_p(z,w) = \sum_{\alpha \in \NN^m} \frac{\partial^\alpha f_p}{\partial w^\alpha}(z,0) w^\alpha.$$
	For each $\alpha \in \NN^m$, define $f_{p,\alpha}:S_p(\sigma) \into \CC$ by $f_{p,\alpha}(z):= \frac{\partial^\alpha f_p}{\partial w^\alpha}(z,0)$.
	It follows from Cauchy's estimates that, for each $\alpha \in \NN^n$, $$\left\| f_{p,\alpha} \right\|_{S_p(\sigma)} \le \frac{\|f_p\|_{S_p(\sigma) \times D(\bar\nu)}}{\nu^{\alpha_1+\cdots+\alpha_m}}.$$  Now fix $\alpha \in \NN^m$.  Then $\sum_{p \in \NN} \|f_{p,\alpha}\|_{S_p(\sigma)}\cdot r^p < \infty$, so the function $f_\alpha: S(\sigma) \into \CC$ defined by $$f_\alpha(z):= \sum_{p \in \NN} f_{p,\alpha}(z)$$ belongs to $\G_{\sigma}$ and satisfies $$\|f_\alpha\|_{\sigma} \le \frac1{\nu^{\alpha_1+\cdots+\alpha_m}}\sum_{p \in \NN} \|f_p\|_{S_p(\sigma) \times D(\bar\nu)}\cdot r^p.$$  Therefore, we have $$\sum_{\alpha \in \NN^m} \|f_\alpha\|_{\sigma} (\bar\mu)^\alpha  = \sum_{\alpha \in \NN^m} \|f_\alpha\|_{\sigma} \mu^{\alpha_1+\cdots+\alpha_m}\le \left(\sum_{p \in \NN} \|f_p\|_{S_p(\sigma) \times D(\bar\nu)}\cdot r^p\right) \cdot \sum_{\alpha \in \NN^m} \left(\frac\mu\nu\right)^{\alpha_1+\cdots+\alpha_m} < \infty,$$ so the series $F := \sum_{\alpha \in \NN^m} f_\alpha X^\alpha \in \G_{\sigma}\left[\!\left[X\right]\!\right]$ is mixed and has polyradius of convergence at least $\bar\mu$.  By uniform convergence and Taylor's Theorem again, it follows that the restriction of $f$ to $S(\sigma) \times D(\bar\nu)$ agrees with $F_{\sigma,\bar\mu}$.
\end{proof}

\begin{cor}  \label{mixed_gev}
	Let $\tau'':= (K,R',r,\mu)$.  Then for $f \in \G_\tau$, the restriction of $f$ to $S(\tau'')$ is definable in $\RR_{\G}$.
\end{cor}

\begin{proof}
	First, note that $z \in S(\tau'')$ if and only if there is a real $(x,\theta) \in S(\sigma) \times D(\bar\nu)$ such that $z = xE(i\theta)$.
	
	Second, if $(x,\theta) \in S(\sigma) \times D(\bar{\nu})$ is real, then by Equation \ref{real_part},
	\begin{align*}
	    f^r(x,\theta) &= \sum_p f_p^r(x,\theta) \\
	        &= \sum_p \Re f_p(xE(i\theta)) \\
	        &= \Re f(xE(i\theta)),
	\end{align*}
	and similarly, by Equation \ref{im_part}, $f^i(x,\theta) = \Im f(xE(i\theta))$.
	
	Third, by \cite[Lemmas 3.5 and 5.1]{RealO-minGamma}, the restrictions to $(S(\sigma) \times D(\bar\nu)) \cap (0,\infty)^m \times \RR^{m+2n}$ of the functions $F^r_{\sigma,\bar\mu}$ and $F^i_{\sigma,\bar\mu}$ obtained in Lemma \ref{mixed_dfbl} are definable in $\RR_{\G}$.
\end{proof}

\begin{proof}[Proof of Theorem \ref{gev_dfbl}]
	Note that $S(\tau)$ is an open neighbourhood of the closure of the set $$\Omega:= S(\tau')\setminus S(\tau''),$$ and recall that $f$ is holomorphic. Using Example \ref{dfbl_exs}(1) and arguing as in the proof of Theorem \ref{cgp_dfbl}, we therefore obtain that the restriction of $f$ to $\Omega$ is definable in $\RR_{\an}$, hence in $\RR_{\G}$.  Together with Corollary \ref{mixed_gev}, this proves the theorem.
\end{proof}

\section{Optimality for the Stirling function}  \label{optimal_sec}

Throughout this section, $\varphi$ denotes the Stirling function introduced in Example \ref{sterling_ex}.  Since the restriction of $\psi \circ \Pi$ to any sector $S_{\bar{\LL}}(\{1\},R,\phi)$, for any $R>0$ and $\phi \in \left(\frac\pi2,\pi\right)$,  belongs to $\G_\tau$ for some $\tau = (\{1\},R,r,\phi)$, we get the following from Theorem \ref{gev_dfbl}: for $\alpha \in (0,\pi)$, set $$S^\infty(R,\alpha):= \{z \in \CC:\ |z|>R, |\arg z| < \alpha\}.$$

\begin{cor} \label{stirling_cor}
	Let $R>0$ and $\alpha \in \left(0,\frac\pi2\right)$.  Then the restriction of $\varphi$ to $S^\infty(R,\alpha)$ is definable in $\RR_{\G}$. \qed
\end{cor}

The next proposition shows that Corollary \ref{stirling_cor} is optimal for definability of the Stirling function on sectors bisected by the positive real half-line.  Recall from \cite[Exercise 5.42]{Sauzin} that $\varphi$ has asymptotic expansion 
$$\hat\varphi(X) = \sum_{k \ge 1} \frac{B_{2k}}{2k(2k-1)} X^{1-2k}$$ at $\infty$,
where the Bernoulli numbers $B_{2k} \in \RR$ are defined such that the convergent series $\sum_{k \ge 1} \frac{B_{2k}}{(2k)!}X^{2k}$ is the Taylor series at 0 of the analytic function $x \mapsto \frac{x}{e^x-1} - 1 + \frac x2$.  

\begin{rmks} \label{stirling_rmk}
	\begin{enumerate}
		\item The series $\hat\varphi$ is divergent and, by \cite[Theorem 5.41]{Sauzin}, $\hat\psi(X) = \hat\varphi(1/X)$ is 1-summable in every direction $d \in \left(-\frac\pi2,\frac\pi2\right)$ with corresponding Borel sum $\psi:\CC \setminus (0,-\infty) \into \CC$.
	  	\item As pointed out in \cite[Exercise 5.46]{Sauzin}, the series $\hat\psi$ is also 1-summable in every direction $d \in \left(\frac\pi2,\frac{3\pi}2\right)$, with corresponding Borel sum $\psi_2:\CC\setminus (0,\infty) \into \CC$.  
	  	\item The function $\varphi_2:\CC\setminus (0,\infty) \into \CC$ defined by $\varphi_2(z):= \psi_2(1/z)$  satisfies
		\begin{equation*}
			\varphi_2(z) = -\varphi(-z) \quad\text{for } z \in \CC\setminus(0,\infty)
		\end{equation*}
		and
		\begin{equation*}
			\varphi(z) - \varphi_2(z) = \sum_{m\ge 1}\frac{e^{-2\pi i k z}}{k} = 	-\log(1-e^{-2\pi i z}) \quad\text{for } \Im z < 0.
		\end{equation*}
		\item Since $\varphi$ and $\psi$ are holomorphic and take real values on $[0,\infty)$, it follows from the Schwartz Reflection Principle that $\varphi(z) = \overline{\varphi(\overline z)}$ and $\psi(z) = \overline{\psi(\overline z)}$ for $z \in \CC \setminus (0,-\infty)$.  
		\item Since the support of $\hat\varphi$ consists of only odd numbers, there is $G \in \Ps{R}{X}$ such that $$\hat\varphi(iX) = iG(1/X),$$
		i.e., the real part of $\hat\varphi(iX)$ is 0. 
	\end{enumerate}
\end{rmks}

\begin{prop} \label{im_not_dfbl}
	For any $a>0$, the restrictions of $\varphi$ to the segments $i(a,\infty)$ and $-i(a,\infty)$ are not definable in $(\RR_{\G},\exp)$.
\end{prop}

\begin{proof}
	 Assume for a contradiction that $a>0$ and the restriction of $\psi$ to $i(0,a)$ is definable in $(\RR_{\G},\exp)$.  
	 	 First, the function $f:(0,a) \into \RR$ defined by 
	 $$f(x):= \Im\psi(ix)$$
	 is then definable in $(\RR_{\G},\exp)$ as well.  Since $\psi$ has asymptotic expansion $\hat\psi$ at 0, the function $\Im\psi$ has asymptotic expansion $\Im\hat\psi$ at 0; hence $f$ has asymptotic expansion $G$ at 0 (as defined in Remark \ref{stirling_rmk}(5)).  It follows from \cite[Corollary 10.10]{RealO-minGamma} that $G$ is $K_1$-summable in the positive real direction, for some finite $K_1 \subseteq (0,\infty)$, and hence that $\hat\psi$ is $K_1$-summable in the direction $\frac\pi2$.
	 
	 Second, by Remark \ref{stirling_rmk}(4), the restriction of $\psi$ to $-i(0,a)$ is definable in $(\RR_{\G},\exp)$ as well.  Therefore, an argument analogous to the above implies that $\hat\psi$ is $K_2$-summable in the direction $-\frac\pi2$, for some finite $K_2 \subseteq (0,\infty)$.  
	 
	 It follows from the above two points that $\hat\psi$ is $K$-summable in every direction (mod $2\pi$), where $K = \{1\} \cup K_1 \cup K_2$.  By \cite[Prop. 13]{Balser:2000fk}, it follows that $\hat\psi$ is convergent, a contradiction.
\end{proof}

Finally, we discuss (non-)definability of $\varphi$ in the left half-plane.
For the next lemma, we define $L:D(1) \into \CC$ by $$L(w):= \log(1-w).$$  Note that $L$ is the sum of a convergent power series with radius of convergence 1, so by Example \ref{dfbl_exs}(1), for each $\delta \in (0,1)$, the restriction of $L$ to $D(\delta)$ is definable in $\RR_{\an}$. Its compositional inverse $E:L(D(1)) \into \CC$ is given by $$E(u) = -\left(e^u-1\right).$$

\begin{lemma} \label{curve_dfbl}
	Let $\gamma:(0,\infty) \into \CC^-$ be a curve such that $\,\lim_{t \to 0} \Re\gamma(t) = -\infty$ and $\,\epsilon:= \liminf_{t \to 0} |\Im\gamma(t)| > 0$, where $$\CC^- := \{z \in \CC:\ \Re z < 0, \Im z \ne 0\}.$$  Let $C:= \gamma((0,\infty))$ be its image, and let $\R$ be any o-minimal expansion of the real field in which the restriction of $L$ to $D(e^{-\pi\epsilon})$ is definable.  Then at most one of $\varphi\rest{C}$ or $\varphi\rest{-C}$ is definable in $\R$.
\end{lemma}

\begin{proof}
	Assume that both $\varphi\rest{C}$ and $\varphi\rest{-C}$ are definable in $\R$ (simply called ``definable'' in this proof); in particular, $C$ is definable, and we may assume that $C$ is connected.  Then either $\liminf_{t \to 0} \Im\gamma(t) = \epsilon$, or $\limsup_{t \to \infty} \Im\gamma(t) = -\epsilon$; by Remark \ref{stirling_rmk}(4), we may assume the latter.  After shrinking $C$ again if necessary, we may then assume that $\left|e^{-2\pi iz}\right| \le e^{-\pi \epsilon} < 1$ for $z \in C$.  
	Therefore, by Remark \eqref{stirling_rmk}(3), the function $f:C \into \CC$ defined by 
	$$f(z) := -L\left(e^{-2\pi iz}\right) = \varphi(z) + \varphi(-z)$$
	is definable.  Since the restriction of $E$ to $L(D(e^{-\pi\epsilon}))$ is also definable, it follows that the function $g:C \into \CC$ defined by $$g(z):= e^{-2\pi iz}$$ is definable.  We leave it to the reader to verify that this contradicts the o-minimality of $\R$.
\end{proof}

\begin{cor} \label{not_dfbl}
	Let $\gamma:(0,\infty) \into \CC^-$ be a curve such that $\,\lim_{t \to 0} \Re\gamma(t) = -\infty$ and $\,\epsilon:= \liminf_{t \to 0} |\Im\gamma(t)| > 0$, where $$\CC^- := \{z \in \CC:\ \Re z < 0, \Im z \ne 0\}.$$  Let $C:= \gamma((0,\infty))$ be its image, and assume also that $C \subseteq \{z \in \CC:|\arg z| > \frac\pi2+\delta\}$ for some $\delta > 0$.  Then the restriction of $\varphi$ to $C$ is not definable in any o-minimal expansion of $\RR_{\G}$.
\end{cor}

\begin{proof}
	Since $-C \subseteq S^\infty(\infty,\frac\pi2-\delta)$, it follows from Corollary \ref{stirling_cor} that the restriction of $\varphi\rest{-C}$ is definable in $\RR_{\G}$.  So by Lemma \ref{curve_dfbl}, $\varphi\rest{C}$ is not definable in any o-minimal expansion of $\RR_{\G}$.
\end{proof}

\begin{rmk}
    The hypothesis that $\liminf_{t \to 0} |\Im\gamma(t)| > 0$ in Lemma \ref{curve_dfbl} and Corollary \ref{not_dfbl} can be dropped when working in an o-minimal structure in which the restriction of $L$ to $D(1)$ is definable. 
\end{rmk}

\section{The \texorpdfstring{$\Gamma$}{Gamma} function} \label{gamma_sec}

We begin this section by describing the sets on which the $\Gamma$ function is definable in $(\RR_{\mathcal{G}},\exp)$. Then we will describe certain regions on which $\Gamma$ cannot be definable in any o-minimal structure. Finally, we show with an example that $\Gamma$ is not the only solution of the difference equation
\[
    f(z+1) = zf(z)
\]
which is definable in $\RR_{\mathcal{G},\exp}$ on an unbounded complex domain.

\subsection{Defining the \texorpdfstring{$\Gamma$}{Gamma} function in \texorpdfstring{$(\RR_{\mathcal{G}},\exp)$}{RG,exp}}
Figure \ref{fig:Gamma} shows two visualizations of the $\Gamma$ function created using the same tool as for Figure \ref{fig:zeta}.

\begin{figure}
    \centering
    \fbox{\includegraphics[width=0.45\linewidth]{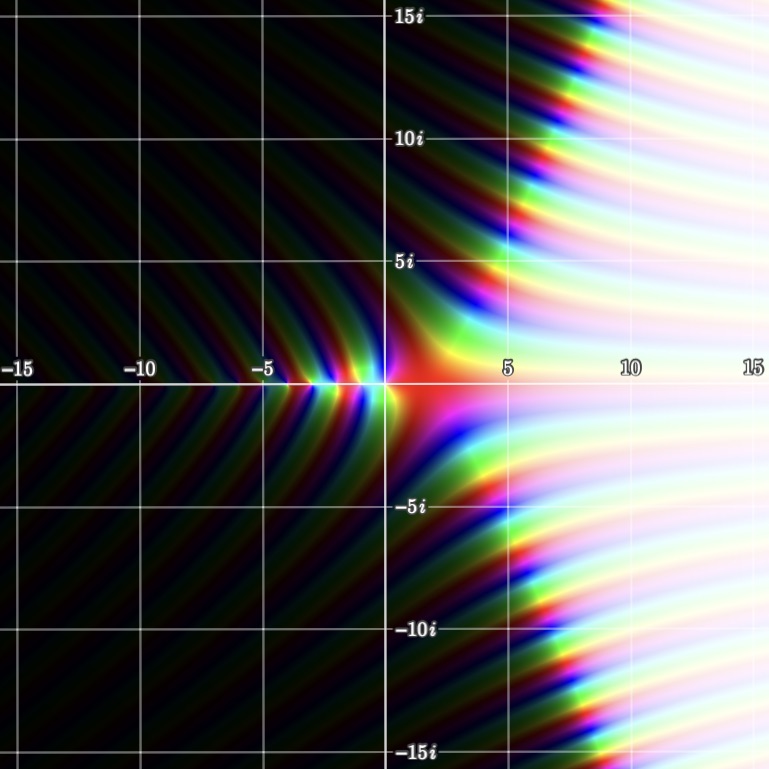}} \hspace{.02\linewidth}
    \fbox{\includegraphics[width=0.45\linewidth]{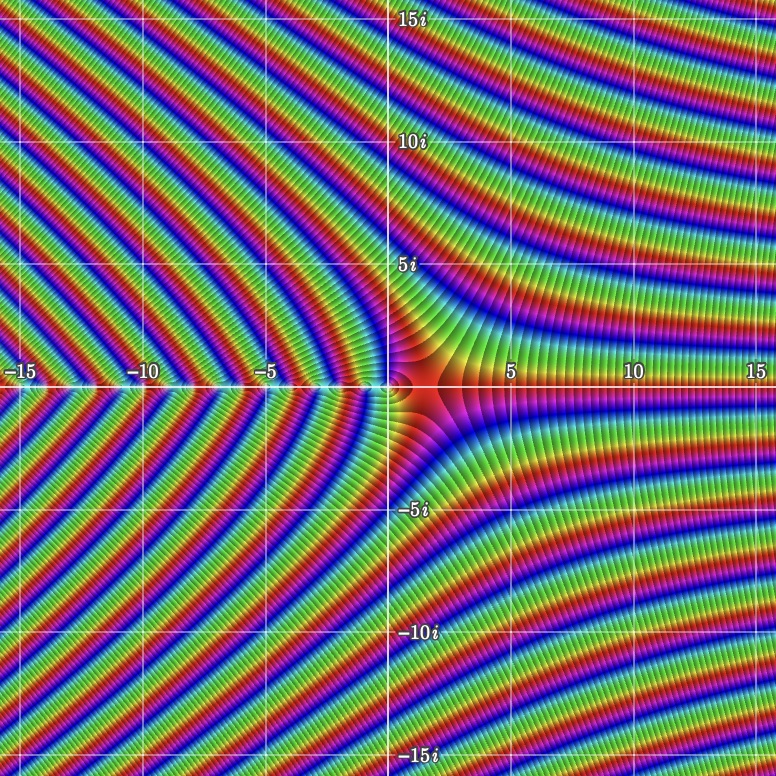}}
    \caption{Two styles of domain colorings for the $\Gamma$ function \cite{DomainColoring}.}
    \label{fig:Gamma}
\end{figure}

Recall from Example \ref{sterling_ex} that
\[
    \Gamma(z) = \sqrt{2\pi}z^{z-\frac{1}{2}}e^{-z}e^{\varphi(z)} = \sqrt{2\pi}e^{\left(z-\frac{1}{2}\right)\log z -z +\varphi(z)}
\]
for $z \in \CC\setminus (-\infty,0]$, where $\varphi(z)$ is the Stirling function.
By Corollary \ref{stirling_cor}, the restriction of $\varphi$ to $S^{\infty}(R,\alpha)$ is definable in $\RR_{\mathcal{G}}$ for any $R>0$ and $\alpha \in \left(0,\frac{\pi}{2}\right)$.
The real and imaginary parts of the complex exponential function are definable in $(\RR_{\mathcal{G}},\exp)$ on domains of the form
\[
    \mathcal{F}_n := \{z \in \CC : 2n\pi \le \Im z < 2(n+1)\pi\}
\]
for $n \in \ZZ$.
So $\Gamma$ restricted to any set of the form
\[
    \widetilde{U_n}(R,\alpha) := \left\{z \in S^{\infty}(R,\alpha) : 2\pi n \le \Im\left(\left(z-\frac{1}{2}\right)\log z - z + \varphi(z)\right) < 2\pi(n+1)\right\}
\]
for $n \in \ZZ$ is definable in $(\RR_{\mathcal{G}},\exp)$. We will write $\widetilde{U_n}$ instead of $\widetilde{U_n}(R,\alpha)$ when $R$ and $\alpha$ are clear from context.

Denote the unique positive real zero of $\Gamma'$ by $x_0 \approx 1.4616$ \cite{uchiyama}.
In Figure \ref{fig:Gamma}, the point $x_0$ is near the center of each image where three red strips meet.
Each set $\widetilde{U_n}(R,\alpha)$ is contained in a rainbow strip bounded between curves along the centers of adjacent red regions. See Figure \ref{fig:enter-label}, right.

\begin{figure}
    \centering
    \fbox{\includegraphics[width=0.45\linewidth]{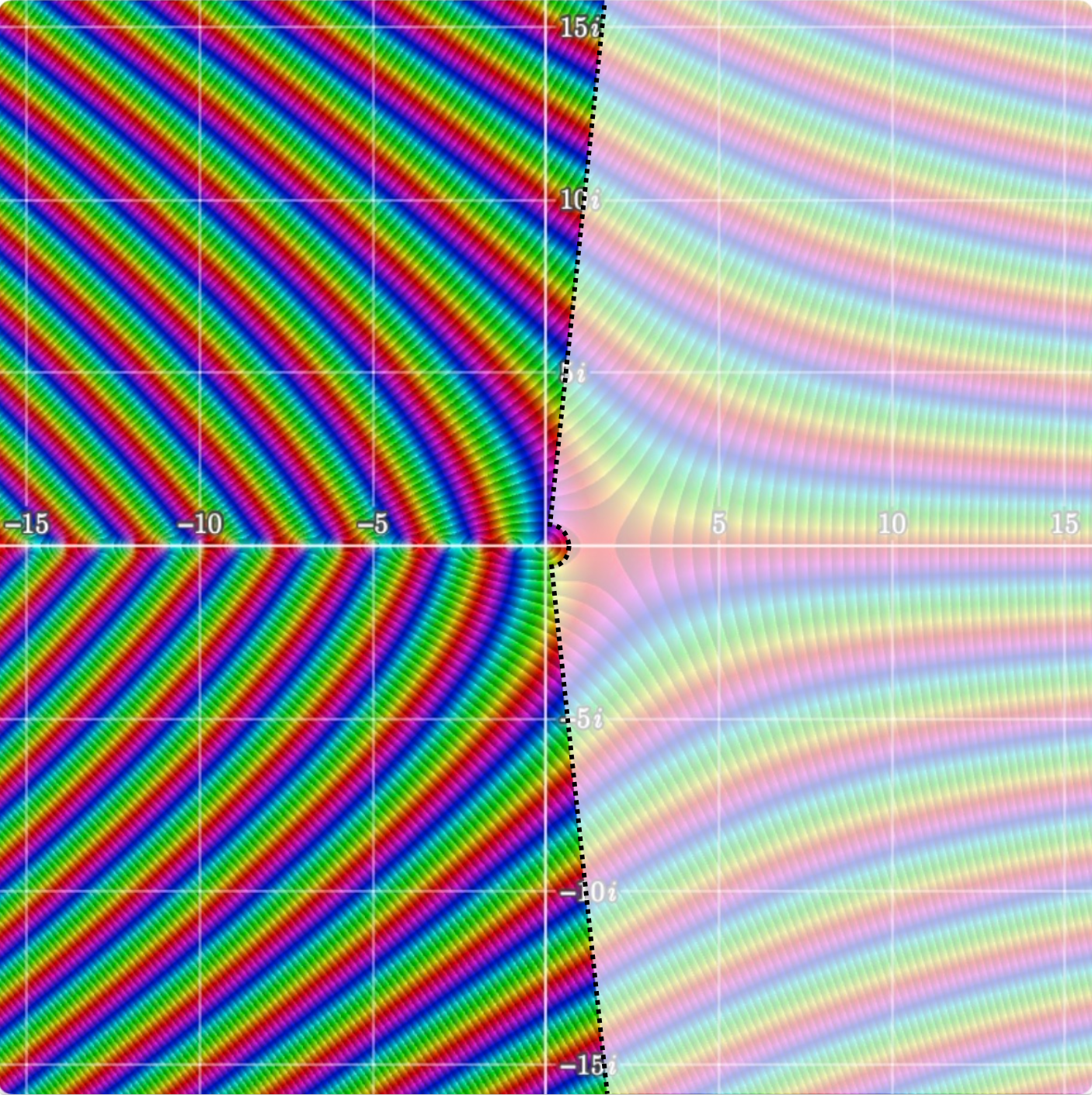}} \hspace{.02\linewidth}
    \fbox{\includegraphics[width=0.45\linewidth]{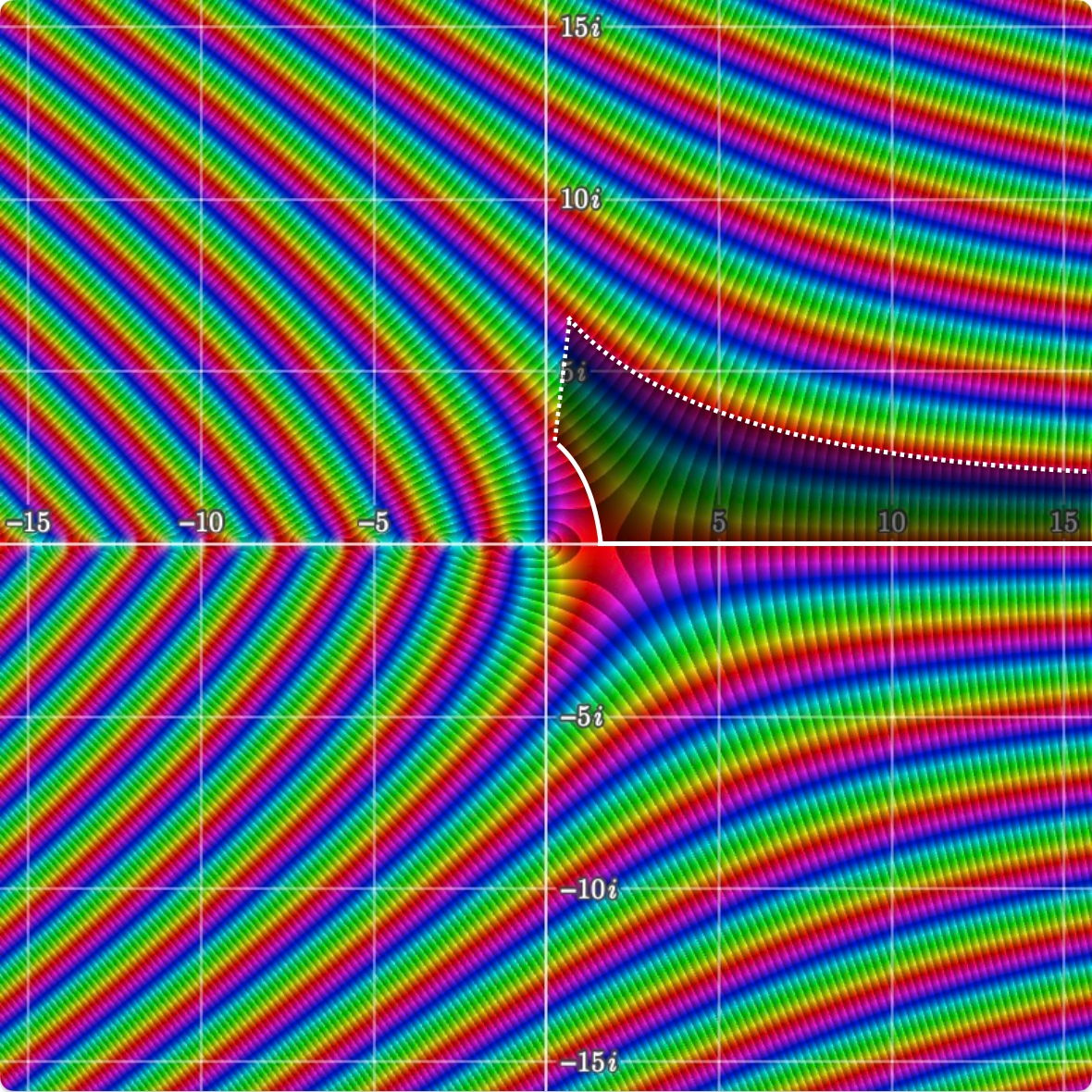}}
    \caption{The region $S^{\infty}\left(\frac{2}{3},\frac{14\pi}{30}\right)$ shaded white and $\widetilde{U_0}$ shaded black.}
    \label{fig:enter-label}
\end{figure}

It is more convenient to provide a qualitative description of the sets
\[
    U_n(R,\alpha) := \widetilde{U_n}(R,\alpha) \cap \{z : \Re z > x_0\}
\]
than to describe the sets $\widetilde{U_n}(R,\alpha)$. Notice that we do not lose much by doing this, as $\widetilde{U_n}(R,\alpha) \setminus U_n(R,\alpha)$ is always bounded and $U_n(R,\alpha) = \widetilde{U_n}(R,\alpha)$ for all but finitely many $n \in \ZZ$.
We will write $U_n$ instead of $U_n(R,\alpha)$ when $R$ and $\alpha$ are clear from context.

Fix $R>0$ and $0<\alpha<\frac{\pi}{2}$. 
In order to describe the sets $U_n$, we study the level curves of $\arg \Gamma$.
Let $A(z)$ be the imaginary part of the exponent of Stirling's formula:
\[
    A(z) := \Im\left(\left(z-\frac{1}{2}\right)\log z - z + \varphi(z)\right).
\]
Then $\arg \Gamma(z) = A(z) \mod 2\pi$.
We will describe the sets defined by $A(z) = \theta$ in the region $\{z : \Re z>x_0\}$ for $\theta \in \RR$.

Since $\Gamma$ is real on the positive real line, $\{z : A(z) = 0, \Re z>x_0\}$ contains the interval $(x_0,\infty)$.
We now recall some facts from \cite{GammaEC}.
Let $C_r := \{z : |\Gamma(z)| = r\}$ for $r \in (0,\infty)$.
\begin{fact}[Propositions 2.5 and 2.7 of \cite{GammaEC}]
\label{fact:fixmod}
    For each $r \in (0,\infty)$, there is a function $y_r(x)$ such that for all large enough $x>x_0$, $|\Gamma(x+iy_r(x))|=r$. The graph of this function is contained in $C_r$ and forms a single $C^1$ curve with positive slope and no horizontal or vertical asymptotes.
    Moreover, \[\frac{d}{dx}(A(x+iy_r(x))) \geq 2(\log(\lfloor x \rfloor)-1)^2. \]
\end{fact}

\begin{fact}[Proposition 2.11 of \cite{GammaEC}]
\label{fact:fixarg}
    For each $\theta \in (-\pi,\pi]$, the set 
    \[
        \{z : \Re z>x_0, \Im z>0, \arg\Gamma(z) = \theta\}
    \]
    is a  collection of disjoint $C^1$ curves, each of which is the graph of a function $y_\theta(x)$ whose slope is negative and approaches zero as $x \to +\infty$.
\end{fact}

Since $\Gamma$ is continuous, since $A(z) = 0$ along the positive real axis, and since $A(z)$ increases along the graph of any $y_r$ by Fact \ref{fact:fixmod}, we must have $A(z)>0$ on $\{z : \Re z > x_0, \Im z >0\}$.
Combining this with Fact \ref{fact:fixarg} shows that for each $\theta>0$,
\[
    A_{\theta} := \left\{z : \Re z>x_0, \Im z>0, A(z)= \theta\right\}
\]
is a curve in the upper right quadrant with negative slope that approaches zero as $\Re z \to \infty$, and if $\theta_1 \ne \theta_2$ then $A_{\theta_1} \cap A_{\theta_2} = \varnothing$.
Since $\Gamma(\overline{z}) = \overline{\Gamma(z)}$, the set
\[
    A_{-\theta} := \left\{z : \Re z>x_0, \Im z<0, A(z)= -\theta \right\}
\]
satisfies $A_{-\theta} = \overline{A_{\theta}}$. So for each $\theta>0$, $A_{-\theta}$ is a curve with positive slope in the lower right quadrant. Moreover, we have $A(z)<0$ on $\{z : \Re z > x_0, \Im z <0\}$. Thus $\{z : A(z) = 0, \Re z>x_0\} = (x_0,\infty)$. Altogether, we have shown the following:

\begin{cor}\label{gamma_dfbl}
    For any $n\in \ZZ$, $\Gamma|_{U_n}$ is definable in $\RR_{\mathcal{G},\exp}$, where $U_n$ is the region in $S^{\infty}(R,\alpha) \cap \{z : \Re z > x_0\}$ bounded between the curves $\{z : A(z) = 2\pi n\}$ and $\{z : A(z) = 2\pi(n+1)\}$. 
\end{cor}

\subsection{Non-definability results}

Next, we prove a non-definability result for $\Gamma$ which complements Proposition \ref{im_not_dfbl} for $\varphi$.

\begin{prop}\label{GammaNonDefinability}
    Let $0 < \epsilon < \frac{\pi}{2}$ and let $\gamma : (0,\infty) \to \{z \in \CC^\times : \epsilon < |\arg z| < \pi - \epsilon\}$ such that $\lim_{t \to \infty} |\Im \gamma(t)| = +\infty$.
    Let $C = \gamma((0,\infty))$.
    Then $\displaystyle \lim_{t \to \infty} |A(\gamma(t))| = \infty$ and $\Gamma|_C$ is not definable in any o-minimal structure.
\end{prop}

\begin{proof}
    Assume for a contradiction that $\Gamma|_{C}$ is definable in some o-minimal expansion $\R$ of the real field.
    In \cite[Chapter 2, Section 4.2]{remmert}, an upper bound $B = B_{\epsilon,M}$ is given on $|\varphi(z)|$ for $|\arg z|<\pi-\epsilon$ and $|z|>M$.
    So writing $\gamma(t) = x_t + iy_t$, we have
    \begin{align*}
        \big|A(\gamma(t))\big| \ge \left|\left(x_t-\frac{1}{2}\right)\arccot\left(\frac{x_t}{y_t}\right)+y_t\left(\log \sqrt{x_t^2+y_t^2}-1\right)\right|-B.
    \end{align*}
    We will show $\lim_{t \to \infty} \left|\left(x_t-\frac{1}{2}\right)\arccot\left(\frac{x_t}{y_t}\right)+y_t\left(\log \sqrt{x_t^2+y_t^2}-1\right)\right| = \infty$, and therefore that $\lim_{t \to \infty} \big|A(\gamma(t))\big| = \infty$ as well. Since $A(z)$ is continuous on $C$ and $\arg \Gamma(\gamma(t)) = A(\gamma(t)) \mod 2\pi$, this would show that, for example, the definable set $\{t \in (0,\infty) : \Re \Gamma(\gamma(t))=0\}$ has infinitely many connected components, which contradicts the o-minimality of $\R$.
    
    We may assume $\lim_{t \to \infty}y_t = +\infty$, as $\Gamma|_C$ is definable if and only if $\Gamma(\overline{z})|_{\overline{C}} = \overline{\Gamma(z)}|_{\overline{C}}$ is definable in $\mathcal{R}$, which holds if and only if $\Gamma|_{\overline{C}}$ is definable in $\mathcal{R}$.
    Since $\arg \gamma(t) < \pi - \epsilon$, we have $y_t \ge |x_t|\tan (\epsilon)$. So 
    \begin{align*}
        \lim_{t \to \infty} &\left|\left(x_t-\frac{1}{2}\right)\arccot\left(\frac{x_t}{y_t}\right)+y_t\left(\log \sqrt{x_t^2+y_t^2}-1\right)\right| \\
            &\ge \lim_{t \to \infty} \left|\left(-|x_t|-\frac{1}{2}\right)(\pi - \epsilon)+|x_t|\tan (\epsilon)\left(\log \sqrt{x_t^2+\left(|x_t|\tan (\epsilon)\right)^2}-1\right)\right| \\
            &\ge \lim_{t \to \infty} \left|\left(-|x_t|-\frac{1}{2}\right)(\pi - \epsilon)+|x_t|\tan (\epsilon)\left(\log |x_t| + \log\sqrt{1+\tan^2(\epsilon)}-1\right)\right| \\
            &\ge \lim_{t \to \infty} |x_t|\left|-\frac{3}{2}(\pi - \epsilon)+\tan(\epsilon)\left(\log |x_t|+ \log\sqrt{1+\tan^2 (\epsilon)}-1\right)\right| \\
            &=\infty. \qedhere
    \end{align*}
\end{proof}

\begin{cor}\label{gamma_optimal}
    Let $X \subset \CC$ and suppose $\Gamma|_X$ is definable in an o-minimal expansion $\mathcal{R}$ of $\RR_{\mathcal{G},\exp}$. Then there must be some $R>0$, $0<\alpha<\frac{\pi}{2}$, and $n \in \NN$ such that 
    \[
        X \setminus \big(U_{-n}(R,\alpha) \cup \cdots \cup U_{n-1}(R,\alpha)\cup U_n(R,\alpha)\big)
    \]
    is bounded.
\end{cor}

\begin{proof}
    By Corollary \ref{sterling_dfbl}, the restriction of the Stirling function $\varphi$ to $S^{\infty}(R,\alpha)$ is definable in $\RR_{\mathcal{G}}$, so also in $\mathcal{R}$, for any $R>0$ and $0 < \alpha < \frac{\pi}{2}$.
    Recall that 
    \[
        \Gamma(z) = \sqrt{2\pi}e^{\left(z-\frac{1}{2}\right)\log z -z +\varphi(z)}
    \]
    for $z \in \CC\setminus\ZZ_{\le 0}$.
    Note that $\left(z-\frac{1}{2}\right)\log z -z +\varphi(z)$ and hence also $A(z)$ are definable in $\mathcal{R}$ on $S^{\infty}(R,\alpha)$.
    The set $A(X)$ must be $i$-bounded because if not, i.e., if the imaginary part of the exponent of Stirling's formula were unbounded on $X$, then
    \[
        \left\{\left(\left(z-\frac{1}{2}\right)\log z -z +\varphi(z),\frac{\Gamma(z)}{\sqrt{2\pi}}\right) : z \in X \right\}
    \]
    would define the graph of the complex exponential function on a region with unbounded imaginary part, which contradicts the o-minimality of $\mathcal{R}$.
    
    Therefore, there exist $n$, $\alpha$ and $R$ such that
    \[
        X \cap S^{\infty}(R,\alpha) \subset \big(U_{-n}(R,\alpha) \cup \cdots \cup U_{n-1}(R,\alpha)\cup U_n(R,\alpha)\big).
    \]
    By Proposition \ref{GammaNonDefinability}, 
    \[
        X \cap \left\{z \in \CC^\times : \frac{\alpha}{2} < |\arg z| < \pi - \frac{\alpha}{2}\right\}
    \]
    must be a bounded set.
    Finally, we claim that set $X \cap -S^{\infty}(R,\alpha)$ must be bounded.
    Suppose it is unbounded. Since $\Gamma$ has a pole at every non-positive integer, either $X^+:= X \cap -S^{\infty}(R,\alpha) \cap \{z : \Im z>0\}$ or $X^-:= X \cap -S^{\infty}(R,\alpha) \cap \{z : \Im z<0\}$ must be unbounded.
    By Remark \ref{stirling_rmk}(3),
    \[
        \varphi(z) + \varphi(-z) = -\log(1-e^{-2\pi i z})
    \]
    for $\Im z <0$.
    Note that $\varphi$ is definable from $\Gamma$ on $X$, and $\varphi$ is definable in $\RR_{\mathcal{G}}$ on $S^{\infty}(R,\alpha)$.
    So $-\log(1-e^{-2\pi i z})$ is definable on $-X^+$.
    If $X^+$ is unbounded, then $-\log(1-e^{-2\pi i z})$ is definable on $-X^+$. Note that the real parts of elements of $-X^+$ are unbounded since $-X^+$ is an unbounded subset of $S^{\infty}(R,\alpha)$. This means the set $\{z \in X : \Im(-\log(1-e^{-2\pi i z}))=0\}$, for instance, is a definable subset in $\mathcal{R}$ with infinitely many components, a contradiction.
    Similarly, if $X^-$ is unbounded, then $-\log(1-e^{-2\pi i z})$ is definable on $X^-$, which has unbounded real part and again gives a contradiction.
\end{proof}

\subsection{Defining another solution to \texorpdfstring{$f(z+1) = zf(z)$}{f(z+1) = zf(z)}  in \texorpdfstring{$(\RR_{\mathcal{G}},\exp)$}{RG,exp}}
We conclude this section by showing that $\Gamma$ is not the only solution of the difference equation $f(z+1)=zf(z)$ definable in $(\RR_{\mathcal{G}},\exp)$ on an unbounded domain.
Consider, for example, $g(z) := \Gamma(z) \left(1-e^{2\pi iz}\right)$. Then $g$ satisfies
\[
    g(z+1) = \Gamma(z+1)\left(1-e^{2\pi i(z+1)}\right) = z\Gamma(z)\left(1-e^{2\pi iz}\right) = zg(z)
\]
on $\CC\setminus \ZZ_{\le 0}$.
Clearly $g$ is also definable in $(\RR_{\mathcal{G}},\exp)$ when restricted to appropriate domains in $\CC$. We will show that these domains are unbounded.
To do this, we qualitatively describe the subsets of the upper left quadrant defined by $A(z) = \theta$ for $\theta \in \RR$. The methods are similar to \cite[Subsection 2.1]{GammaEC} in which the behavior of $\Gamma$ in the upper right quadrant is studied, but the information we need does not directly follow from the results there.

\begin{lemma}\label{ModIncrXDecrY}
    For each $y \ge 2$, the map $x\mapsto |\Gamma(x+iy)|$ is injective with non-vanishing derivative. For each $x \in \RR$, the map $y\mapsto |\Gamma(x+iy)|$ is injective with non-vanishing derivative on $y > 0$. 
    Moreover, $|\Gamma(x+iy)|$ grows exponentially to $+\infty$ as $x \to +\infty$ and decays exponentially to zero as $|y| \to +\infty$
\end{lemma}
\begin{proof}
    Write $|\Gamma(x+iy)| = \left|\frac{\exp(-\gamma (x+iy))}{x+iy}\prod_{n=1}^{\infty}\left(1+\frac{x+iy}{n}\right)^{-1}\exp\left(\frac{x+iy}{n}\right)\right|$ as an infinite product where $\gamma \approx 0.577$ is the Euler-Mascheroni constant.
    Recall that for a differentiable product of differentiable functions $f(x) = \prod_{n=0}^{\infty}f_{n}(x)$ we have $f'(x) = f(x)\sum_{n=0}^{\infty}\frac{f_{n}'(x)}{f_{n}(x)}$.
    So we compute the derivative of $x\mapsto |\Gamma(x+iy)|$ and show it is positive for all $x$ and all $y \ge 2$:
    \begin{align*}
        \frac{\partial}{\partial x}\left(|\Gamma(x+iy)|\right)
            &= \frac{\partial}{\partial x}\left(\frac{\exp(-\gamma x)}{\sqrt{x^2+y^2}}\prod_{n=1}^{\infty} \frac{\exp\left(\frac{x}{n}\right)}{\sqrt{\left(1+\frac{x}{n}\right)^2+\left(\frac{y}{n}\right)^2}}\right) \\
            &= |\Gamma(x+iy)|\left(-\gamma - \frac{x}{x^2+y^2}+\sum_{n=1}^{\infty} \left(\frac{1}{n}-\frac{n+x}{(n+x)^2+y^2}\right)\right) \\
            &\ge |\Gamma(x+iy)|\left(-\gamma - \frac{x}{x^2+4}+\sum_{n=1}^{\infty} \left(\frac{1}{n}-\frac{n+x}{(n+x)^2+4}\right)\right) \\
            &> |\Gamma(x+iy)|\left(-\gamma - \frac{1}{4} + \left(1-\frac{1}{4}\right)+\left(\frac{1}{2}-\frac{1}{4}\right)+\sum_{n=3}^{\infty} \left(\frac{1}{n}-\frac{n}{n^2+4}\right)\right) \\
            &>|\Gamma(x+iy)|\left(-\gamma + \frac{3}{4}\right)
    \end{align*}
    which is strictly positive since $|\Gamma(x+iy)|$ never vanishes.
    Next, we compute that the derivative of $y \mapsto |\Gamma(x+iy)|$ is negative for any $x \in \RR$: 
    \begin{align*}
        \frac{\partial}{\partial y}\left(|\Gamma(x+iy)|\right)
            = |\Gamma(x+iy)|\left(-\frac{y}{x^2+y^2} - \sum_{n=1}^{\infty}\frac{y}{(x+n)^2+y^2}\right) 
            <0
    \end{align*}

    Now we consider $|\Gamma(x+iy)|$ as $x \to +\infty$ and as $|y| \to +\infty$. For each $\theta \in \left(\frac{\pi}{2},\pi\right)$, there is $M_{\theta} >0$ such that if $|\arg z| \le \theta$ and $|z|>M_{\theta}$, then $|\varphi(z)| < 1$, so $\frac{1}{e}\leq|\exp(\varphi(z))|\leq e$, where $e=\exp(1)$. See \cite[Chapter 2, Section 4.2]{remmert}.
    So for $|\arg z| \le \theta$ and $|z|> M_{\theta}$ we have
    \begin{multline*}
        \frac{\sqrt{2\pi}}{e}\exp\left(\left(x-\frac{1}{2}\right)\log|z|- y\arg(z) - x\right) \leq|\Gamma(z)|  \\
            \leq e\sqrt{2\pi}\exp\left(\left(x-\frac{1}{2}\right)\log|z|- y\arg(z) - x\right).
    \end{multline*}
    It follows that $|\Gamma(x+iy)|$ tends exponentially to zero as $|y|$ tends to $+\infty$, and $|\Gamma(x+iy)|$ tends exponentially to $+\infty$ as $x$ tends to $+\infty$. 
\end{proof}

\begin{lemma}\label{FixArgCurves}
    For each $\theta\in\RR$, there is a function $y_{\theta}(x)>2$ and $r_{\theta} \in \RR$ such that for all $x< r_{\theta}$, we have $A(x+iy_{\theta}(x)) = \theta$ and the graph of $y_{\theta}$ is a single $C^1$ curve with negative slope and no vertical asymptotes.
\end{lemma}

\begin{proof}
    We first use Lemma \ref{ModIncrXDecrY} to describe the $|\Gamma|$-level curves. This information will help us describe the $\arg \Gamma$-level curves because $\Gamma$ is a conformal map.
    Let $a,b \in \RR$ with $b >2$. By Lemma \ref{ModIncrXDecrY}, $\displaystyle \frac{\partial |\Gamma(x+iy)|}{\partial y}(a + ib) \ne 0$. By the implicit function theorem, there is a unique $C^1$ function $y(x)$ such that $\mathrm{graph}(y)=\{z : |\Gamma(z)| = |\Gamma(a+ib)|\}$ in a neighborhood of $a+ib$. Also by Lemma \ref{ModIncrXDecrY}, $|\Gamma(x+iy)|$ strictly decreases as $y \to +\infty$ and strictly increases as $x \to +\infty$. So by the intermediate value theorem, $y'(a+ib) >0$.

    Let $\theta = A(a+ib)$, so that $\arg \Gamma(a+ib) = \theta \mod 2\pi$, and let $A_{\theta} = \{z : A(z) = \theta\}$.
    Since $\Gamma'|_{\CC \setminus (-\infty,x_0]}$ does not vanish, $\Gamma$ is conformal at $a+ib$. So there is a neighborhood $U$ of $a+ib$ such that $U \cap A_{\theta}$ is a curve $C$ which intersects $\mathrm{graph}(y)$ at $a+ib$ at a right angle. 
    Since $y'(a+ib)>0$, $C$ must pass through $a+ib$ with negative slope. In particular, $\frac{\partial \arg\Gamma(x+iy)}{\partial y}(a+ib) \ne 0$ so we can apply the implicit function theorem to obtain a unique $C^1$ function $y_\theta(x)$ such that $\mathrm{graph}(y_{\theta}) = A_{\theta}$ in a neighborhood of $a+ib$ and $y_{\theta}'(a+ib)<0$.

    The above argument shows $\frac{\partial \arg\Gamma(x+iy)}{\partial y}<0$ on $\{x+iy : y>2\}$. So the only barrier to extending the domain of $y_{\theta}$ on the left to $(-\infty,a)$ is if $y_{\theta}$ has a vertical asymptote.
    But if $y_{\theta}$ had a vertical asymptote, Proposition \ref{GammaNonDefinability} would imply that $A(z)$ is unbounded along $\mathrm{graph}(y_{\theta})$, which contradicts that $y_{\theta}$ is contained in $A_{\theta}$.
    So $y_{\theta}$ cannot have any vertical asymptotes, and the domain of $y_{\theta}$ can be extended to $(-\infty,a)$.
    Similarly, the domain of $y_{\theta}$ can be extended on the right unless $\mathrm{graph}(y_{\theta})$ intersects the line $y=2$. Let $r_{\theta}$ be the real part of this point of intersection if it exists, or $+\infty$ otherwise.

    Finally, we show that $A_{\theta} \cap \{x+iy : y>2\}$ consists of a single $C^1$ curve. Suppose toward a contradiction that $a^*+ib^* \in A_{\theta} \setminus \mathrm{graph}(y_{\theta})$ and $b^*>2$. Then there is a function $y_{\theta}^* : (-\infty,r_{\theta}^*) \to A_{\theta}$ whose graph contains $a^*+ib^*$.
    The graphs of $y_{\theta}$ and $y_{\theta}^*$ do not intersect because $A_{\theta}$ is locally the graph of a function, so without loss of generality, suppose $y_{\theta}(x) < y_{\theta}^*(x)$ for all $x <r_{\theta}$.
    Let $C$ be a $|\Gamma|$-level curve that intersects $\mathrm{graph}(y_{\theta})$ at some point $z_0$. 
    We claim $C$ also intersects $\mathrm{graph}(y_{\theta}^*)$. If not, then $C$ must approach a horizontal asymptote as $x$ tends to $+\infty$ because its slope is positive and $\mathrm{graph}(y_{\theta}^*)$ has negative slope. But by Fact \ref{fact:fixmod}, $C$ does not approach a horizontal asymptote in the upper right quadrant. So $C$ intersects the graph of $y_{\theta}^*$ at $z_1$, say. 
    Now let $\gamma: [0,1] \to C$ be a $C^1$ function parametrizing $C$ between $z_0$ and $z_1$. Then $A(\gamma(0)) = A(\gamma(1))$, and Rolle's theorem implies that $A'(\gamma(s))=0$ for some $0<s<1$. But then $\Gamma'(\gamma(s)) = 0$, since $C$ is a $|\Gamma|$-level curve and $(|\Gamma(\gamma(t))|)' = 0$ for all $t \in (0,1)$. This cannot be since all the zeroes of $\Gamma'$ lie along the real axis.
    So we must have $A_{\theta} \cap \{x+iy : y > 2\} = \mathrm{graph}(y_{\theta})$.
\end{proof}

\begin{theorem}
    The function $\Gamma(z)(1-e^{2\pi i z})$ is definable in $(\RR_{\mathcal{G}},\exp)$ on an unbounded complex domain.
\end{theorem}

\begin{proof}
    Let $R>0$ and $0 < \alpha < \frac{\pi}{2}$.
    By Remark \ref{stirling_rmk}(3) the function $-\varphi(z) = \varphi(-z) +\log(1-e^{-2\pi i z})$ is definable in $\RR_{\mathcal{G}}$ on $S^{\infty}(R,\alpha) \cap \{z : \Im z<0\}$, or equivalently, the function $-\varphi(-z) = \varphi(z) +\log(1-e^{2\pi i z})$ is definable on $-S^{\infty}(R,\alpha) \cap \{z : \Im z>0\}$.
    For $\Im z > 0$, define
    \[
        g(z) := \sqrt{2\pi}e^{\left(z-\frac{1}{2}\right)\log z-z-\varphi(-z)} = \sqrt{2\pi}e^{\left(z-\frac{1}{2}\right)\log z-z+\varphi(z) +\log(1-e^{2\pi i z}))} = \Gamma(z)(1-e^{2\pi i z}).
    \]
    Then $g$ is definable in $(\RR_{\mathcal{G}},\exp)$ on the domain
    \[
        V_n := \left\{z \in -S^{\infty}(R,\alpha) : \Im z>0 \text{ and }2\pi n \le \Im\left(\left(z-\frac{1}{2}\right)\log z - z -\varphi(-z)\right) < 2\pi(n+1) \right\}
    \]
    for each $n \in \ZZ$.

    We now show that each $V_n$ is unbounded.
    For $\theta \in \RR$, define $B_{\theta}$ to be the set given by fixing the imaginary part of the exponent in the definition of $g$ to be $\theta$:
    \[
        B_{\theta} := \left\{z : \Im\left(\left(z-\frac{1}{2}\right)\log z-z-\varphi(-z)\right) = \theta\right\}.
    \]
    If $z \in B_{\theta}$, then $\arg g(z) = \theta \mod 2\pi$.
    Recall from Lemma \ref{FixArgCurves} that for each $\theta \in \RR$, $A(z)=\theta$ defines a single unbounded curve $A_{\theta}$ with negative slope in the region $\{z : \Im z >2\}$. 
    Note that $\arg g(z) = \arg \Gamma(z) + \arg(1-e^{2\pi i z})$. Also, $|e^{2\pi i z}| < e^{-4\pi}$ and $|\arg(1-e^{2\pi i z})|<2e^{-4\pi}$ for $\Im z > 2$.
    So in the region $\{z : \Im z>2\}$, $B_{\theta}$ is bounded between the curves $A_{\theta-2e^{-4\pi}}$ and $A_{\theta+2e^{-4\pi}}$.
    By Lemma \ref{GammaNonDefinability}, the intersection $A_{\theta} \cap \{z \in \CC^{\times} : \epsilon < |\arg z| < \pi - \epsilon\}$ must be a bounded set for any $\epsilon > 0$.
    So all but possibly a bounded piece of each curve $A_{\theta}$ is contained in $-S^{\infty}(R,\alpha)$.
    Thus each $V_n$ is an unbounded set.
\end{proof}

\bibliographystyle{plain}
\bibliography{bibliography}

\end{document}